\documentclass[11pt]{amsart}
\usepackage[svgnames,dvipsnames]{xcolor}
\usepackage[hmargin=2.5cm,vmargin=2.5cm]{geometry}
\usepackage{amssymb,amsmath,amsthm,enumerate}
\usepackage{hyperref}
\usepackage{soul}

\newcommand{\ie}{\textit{i.e. }}

\newcommand{\na}{\nabla}
\newcommand{\pa}{\partial}
\newcommand{\xd}{\mathrm{d}}
\newcommand{\eps}{\varepsilon}

\newcommand{\bD}{\mathbb{D}}
\newcommand{\bN}{\mathbb{N}}
\newcommand{\bP}{\mathbb{P}}
\newcommand{\bR}{\mathbb{R}}

\renewcommand{\AA}{{\boldsymbol{A}}}
\newcommand{\BB}{{\boldsymbol{B}}}

\newcommand{\FF}{{\boldsymbol{F}}}
\newcommand{\GG}{{\boldsymbol{G}}}
\newcommand{\KK}{{\boldsymbol{K}}}
\newcommand{\MM}{{\boldsymbol{M}}}
\newcommand{\PP}{{\boldsymbol{P}}}
\newcommand{\QQ}{{\boldsymbol{Q}}}
\newcommand{\WW}{{\boldsymbol{W}}}
\newcommand{\ff}{{\boldsymbol{f}}}
\renewcommand{\gg}{{\boldsymbol{g}}}
\newcommand{\hh}{{\boldsymbol{h}}}
\newcommand{\mm}{{\boldsymbol{m}}}
\newcommand{\nn}{{\boldsymbol{n}}}
\newcommand{\nnu}{{\boldsymbol{\nu}}}

\newcommand{\kki}{{\boldsymbol{\chi}}}
\newcommand{\M}{(\nn \MM)}

\newcommand{\im}{\operatorname{im}}
\newcommand{\Span}{\operatorname{Span}}

\newcommand{\mI}{\mathcal{I}}
\newcommand{\mL}{\boldsymbol{\mathcal{L}}}
\newcommand{\mN}{\boldsymbol{\mathcal{N}}}
\newcommand{\hL}{\mathbf{\bar\mL}}
\newcommand{\iL}{\hL^{-1}}

\newcommand{\CbdV}{C_{V}} 
\newcommand{\Cbi}{{C_{\mathrm{inv}}}} 
\newcommand{\sgd}{\lambda} 

\newcommand{\Cm}{C_{\mm}} 
\newcommand{\Gnm}{\gamma_{\theta}} 
\newcommand{\Cnmbet}{C_{\beta}} 
\newcommand{\Cnbet}{C_{\nn,\beta}} 
\newcommand{\CkerL}{C_{\mathrm{eq}}} 
\newcommand{\KkerL}{K_{\mathrm{eq}}} 
\newcommand{\Cki}{C_\chi}
\newcommand{\Cth}{C_\theta}
\newcommand{\sass}{\varepsilon}

\allowdisplaybreaks

\newtheorem{theorem}{Theorem}
\newtheorem{corollary}[theorem]{Corollary}
\newtheorem{lemma}[theorem]{Lemma}
\newtheorem{proposition}[theorem]{Proposition}
\newtheorem*{ass}{Assumption}
\newtheorem{remark}{Remark}

\begin{document}
\title[Energy method for the Boltzmann equation of gaseous mixtures]{Energy method for the  Boltzmann equation \\ of monatomic gaseous mixtures}
\thanks{This work was supported by the European COST Action CA18232 Mat-Dyn-Net, the joint Austrian-French-Serbian Danube PHC project No. 45286PC, and by the French CNRS PICS08057 and IEA00375 projects. S.~S. and M.~P.-\v C. acknowledge  support of the Ministry of Education, Science and Technological Development of the Republic of Serbia (Grant No. 451-03-9/2021-14/200125).  M.~P.-\v C.'s work is supported by the Alexander-von-Humboldt Foundation.
}

\author[L. Boudin]{Laurent Boudin}
\address{LB: Sorbonne Université, CNRS, Université Paris Cité, Laboratoire Jacques-Louis Lions (LJLL), F-75005 Paris, France}
\email{laurent.boudin@sorbonne-universite.fr}

\author[B. Grec]{B\'er\'enice Grec}
\address{BG: MAP5, CNRS UMR 8145, Universit\'e de Paris, F-75006 Paris, France}
\email{berenice.grec@u-paris.fr}

\author[M. Pavi\'c-\v Coli\'{c}]{Milana Pavi\'c-\v Coli\'{c}}
\address{MPC: Department of Mathematics and Informatics, Faculty of Sciences, University of Novi Sad, Trg Dositeja Obradovi\'ca 4, 21000 Novi Sad, Serbia  } 
\email{milana.pavic@dmi.uns.ac.rs}

\author[S. Simi\'c]{Srboljub Simi\'c}
\address{SS: Department of Mathematics and Informatics, Faculty of Sciences, University of Novi Sad, Trg Dositeja Obradovi\'ca 4, 21000 Novi Sad, Serbia} 
\email{srboljub.simic@dmi.uns.ac.rs}

\bibliographystyle{abbrv}

\begin{abstract}
In this paper, we present an energy method for the system of Boltzmann equations in the multicomponent mixture case, based on a micro-macro decomposition. More precisely, the perturbation of a solution to the Boltzmann equation around a global equilibrium is decomposed into the sum of a macroscopic and a microscopic part, for which we obtain {\it a priori} estimates at both lower and higher orders. These estimates are obtained under a suitable smallness assumption. The assumption can be justified {\it a posteriori} in the higher-order case, leading to the closure of the corresponding estimate. 
\end{abstract}

\maketitle

\section{Introduction}

In the last decade or so, the Boltzmann equation for mixtures, which was already mentioned in \cite{cercibook},  attracted the attention of many works. The modelling issue, for both monatomic and polyatomic gases, was for instance discussed in \cite{des-mon-sal, bar-bis-bru-des, simic-parma} (see also the references therein). Many works focused on the analysis of the monatomic case, like \cite{nous, dau-jun-mou-zam, bri-dau, bon-bou-bri-gre}, which were dedicated to compactness, hypocoercivity-related and stability results. Well-posedness and regularity were investigated in \cite{bri-dau, bri, gam-pav, MPC-A-G, can-gam-pav}, and asymptotics questions were tackled in \cite{bou-gre-sal,hut-sal, bou-gre-pav, bon-bri}. 

Some of the previous papers, for instance \cite{bri-dau}, rely on the so-called micro-macro decomposition. In the present work, we aim to provide a more detailed insight on that decomposition in the mixture case. Indeed, the micro-macro decomposition of a solution of the linearized Boltzmann equation has a key role in the study of both mathematical and numerical properties of that solution.  It was introduced for the monospecies Boltzmann equation in \cite{guo03, guo04} on the one hand, and in \cite{liu-yan-yu,liu-yu04,yan-zha} on the other hand. The method consists in considering the equilibrium perturbation as the sum of a macroscopic part and a microscopic one. The macroscopic part can be decomposed on a finite-dimensional subspace, where the associate coordinates solve some conservation laws of fluid type, whereas the microscopic one still solves a kinetic equation. Nevertheless, the microscopic part is incorporated in macroscopic conservation laws and fills the gap between the usual Navier-Stokes approximation and the complete kinetic equation \cite{liu-yan-yu}. In fact, it brings information which is essential to provide proper estimates of the perturbed solutions of the kinetic equation.

In the monospecies case, the micro-macro decomposition and the underlying energy method were used for hypocoercivity estimates, see \cite{dol-mou-sch}, for large-time behaviour studies \cite{liu-yan-yu-zha, lee-liu-yu, mou-neu} (see also \cite{liu-liu} for a binary mixture), for propagation of one-dimensional waves \cite{liu-yu07}, or to obtain Green's function for the Boltzmann equation \cite{liu-yu11}. In numerical analysis, this decomposition is a major tool to build asymptotic-preserving (AP) schemes, allowing to pass, for instance, from the Boltzmann equation to the Navier-Stokes equations \cite{ben-lem-mie}, to exactly conserve some physical quantities \cite{gam-jin-liu}, or to quantify uncertainty in kinetic equations \cite{dim-par}. As far as the mixture case is concerned, note that several attempts relying on a micro-macro decomposition were already performed, with a BGK approximation \cite{jin-shi, jin-li}, in the two-species case \cite{cre-kli-pir}, and in the general case with partial results \cite{bri-dau}.

In this paper, we study the micro-macro decomposition and the corresponding energy method  in the multicomponent mixture case, by following the strategy of \cite{liu-yu04}.  More precisely, we start from an equilibrium not depending on the time and space variables. The perturbation of this equilibrium is then decomposed into the sum of microscopic and macroscopic parts, for which we obtain lower and higher order estimates, first using relevant smallness assumptions, and exhibiting closure in the one-dimensional (in space) setting.

It is peculiar for the mixture, in contrast to the monospecies case, that the microscopic part contributes to macroscopic equations not only in the momentum and energy conservation laws, but also in the  mass conservation law. This effect is crossed with the perturbation of the energy variable, which altogether makes the procedure of finding proper estimates much more involved. This  problem is solved by means of introducing a suitable fluid quantity.
\\

The paper is structured as follows.  In the next section, we give a preliminary overview describing  the framework for the subsequent analysis. Then,  in Section \ref{s:mainideas}, we discuss the methodology and main ideas relying on the micro-macro decomposition, and state our main results. They are {\it a priori} estimates on the perturbation based on the decomposition, whose proofs are exposed in Sections \ref{s:loe} and \ref{s:hoe}. In particular, we provide very detailed explanations for the lower-order estimate, knowing that, for the higher-order one, the same kind of computations and ideas are developed. 

\section{Preliminaries}

We consider an ideal gas mixture constituted with $I \geq 2$ monatomic species. Each species, indexed by $1\le i \le I$, is described thanks to a distribution function $F_i$, which is nonnegative, and depends on time $t\in [0,T]$, $T>0$, space position $x\in\bR$ and microscopic velocity $v\in\bR^3$. We denote by $m_i$ the atomic mass of species $i$. We emphasize that we choose to work here in a one-dimensional setting for the space variable $x$, not only for the sake of simplicity. Indeed, if most computations and results remain true in dimensions $2$ and $3$, the estimates are closed in this work by introducing the antiderivative of the macroscopic part of the decomposition, which can only be performed in a one-dimensional setting.

To consider the species altogether, we introduce the vector distribution function of the mixture, denoted by $\FF=(F_i)_{1\le i \le I}$. It satisfies the system of Boltzmann equations, also written in a vector form,
\begin{equation}\label{BEvectorform}
\pa_t \FF + v_1 \pa_x \FF = \QQ(\FF,\FF),
\end{equation}
where $v_1$ is the coordinate of velocity $v$ in a direction of the space variable $x$, and $\QQ$ is the vector collision operator, which only acts on the velocity variable $v$. The vector collision operator $\QQ$ can be defined component-wise. To this end, we first need to recall the microscopic context of the collisions. 

We assume that the mixture only involves elastic collisions, without chemical reactions. Consider two colliding molecules, one of species $i$ and another one of species $j$, with respective pre-collisional velocities $v'$ and $v'_*$. Those velocities change after collision into post-collisional velocities $v$ and $v_*$, with both momentum and kinetic energy conserved, \ie
\begin{equation} \label{loi_de_conservation_micro}
m_i v' + m_j v'_* = m_i v + m_j v_* , \qquad \frac{1}{2} m_i {v'}^{2} + \frac{1}{2} m_j {v_*'}^{2} = \frac{1}{2} m_i v^{2} + \frac{1}{2} m_j {v_*}^{2}. 
\end{equation}
The previous equalities allow to introduce a parameter $\omega\in \mathbb S^2$, enabling to write $v'$ and $v'_*$ in terms of $v$ and $v_*$ as
\begin{equation}\label{collisionrule}
v' = \dfrac{m_i v + m_j v_*}{m_i + m_j} + \dfrac{m_j}{m_i + m_j} T_{\omega} (v-v_*), \qquad  
v'_* = \dfrac{m_i v + m_j v_*}{m_i + m_j} - \dfrac{m_i}{m_i + m_j} T_{\omega} (v-v_*),
\end{equation}
denoting $T_{\omega} z= z-2(\omega \cdot z) \omega$ for any $z\in\bR^3$.

Then, for any $i$, $j$, we can define the operator $Q_{ij}$ describing the atomic interactions of species $i$ with species $j$. It only acts on the velocity variable and is given by
$$Q_{ij} (f_i, g_j) (v) = \iint_{\bR^3 \times S^2} \left[ f_i(v') g_j(v'_*) - f_i (v) g_j (v_*) \right] {\mathcal{B}}_{ij} (v, v_* ,\omega)\,  \xd \omega \, \xd v_*,
$$
for any species-related real-valued functions $f_i$, $g_j$ of the velocity variable.
The cross-section $\mathcal{B}_{ij}$ allows to classify the way species $i$ and $j$ interact and must satisfy the micro-reversibility property
\begin{equation*}
\mathcal{B}_{ij} (v', v'_* ,\omega) = \mathcal{B}_{ji} (v_*, v ,\omega) = \mathcal{B}_{ij} (v, v_* ,\omega)\geq 0.
\end{equation*}
Moreover, in this work, we make the hard-sphere assumption, for any $i$, $j$, 
\begin{equation}\label{ass Bij L2}
	\mathcal{B}_{ij} (v, v_* ,\omega)  = \beta_{ij}\left| \left(v-v_*\right) \cdot \omega \right|,
\end{equation}
where $\beta_{ij}>0$ is given. The assumption is required to ensure needed properties of the collision frequency and to deal with the nonlinearity. 

Eventually, we can define the $i$-th component of $\QQ$, with $\ff=(f_j)_{1\le j \le I}$, $\gg=(g_j)_{1\le j \le I}$, by
$$Q_i(\ff,\gg) = \sum_{j=1}^I Q_{ij}(f_i, g_j).$$

\medskip

Before recalling the main properties of the solutions to \eqref{BEvectorform}, let us introduce some very convenient notations. First, we define a component-wise product of two vectors $\AA=(A_i)_{1\le i \le I}$, $\BB=(B_i)_{1\le i \le I}$ and a vector-valued function of $\AA$, for  $\Phi:\bR\to\bR$, by
$$\AA \BB = \begin{pmatrix} A_1 \, B_1 \\  A_2 \, B_2 \\ \vdots \\ A_I \, B_I \end{pmatrix}, \qquad \Phi(\AA) = \begin{pmatrix} \Phi(A_1) \\ \Phi(A_2) \\ \vdots \\ \Phi(A_I) \end{pmatrix}.$$
This way, we can write, for instance, $\AA^{1/2} = ({A_i}^{1/2})_{1\le i \le I}$, when  $A_i \geq 0$. Finally, $L^2(\mathbb{R}^3)^I$ is endowed with its natural scalar product and norm, \ie we set, for any vector functions $\ff=(f_i)_{1\le i \le I}$, $\gg=(g_i)_{1\le i \le I}$ $\in L^2(\mathbb{R}^3)^I$,
$$\langle \ff,\gg \rangle_I = \sum_{i=1}^I \int_{\bR^3} f_i \, g_i \, \xd v, \qquad \|\ff\|_I = {\langle \ff, \ff \rangle_I}^{1/2}.$$

Conservative properties of the Boltzmann equations are obtained thanks to the weak form of the collision operator that uses some symmetries built in the model. In the mixture setting, the weak form is  carefully described, for example, in \cite{des-mon-sal,bou-gre-sal, bou-gre-pav}. We only mention here the final formula. For any functions $\GG$ and $\boldsymbol{\psi}$ for which it makes sense, we have
\begin{multline*}
\langle \QQ(\GG,\GG), \boldsymbol{\psi} \rangle_I = - \frac 14 \sum_{i,j=1}^I \iiint_{\bR^3\times \bR^3\times S^2} \left[ G_i(v') G_j(v'_*) - G_i (v) G_j (v_*) \right] \\ \times \left[ \psi_i(v') + \psi_j(v'_*) - \psi_i(v) - \psi_j(v_*)\right] \mathcal{B}_{ij}(v,v_*,\omega) \, \xd \omega \, \xd v_* \,\xd v.
\end{multline*}

In this paper, we work in a perturbative setting, around a global equilibrium distribution function. Its notion is introduced in the so-called $H$-theorem, see \cite{des-mon-sal} for instance. Let us first define the  entropy production functional
$$D(\GG)= \langle \QQ(\GG,\GG), \log \GG \rangle_I.$$
The $H$-theorem reads
\begin{proposition} \label{htheo}
	Assume that all the cross sections are positive almost everywhere and that $\GG$ is such that both $\QQ(\GG,\GG)$ and $D(\GG)$ are well defined. Then
	\begin{enumerate}
		\item[(a)] The entropy production is non-positive, {\normalfont i.e.} $D(\GG)\le 0$.
		\item[(b)] Moreover, the three following properties are equivalent:
		\begin{enumerate}
			\item[ i. ] for any $1 \le i, j \le I$, $Q_{ij} (G_i, G_j)= 0$;
			\item[ ii. ] the entropy production vanishes, that is $D(\GG) = 0$;
			\item[ iii. ] there exist $T>0$ and $u \in \mathbb{R}^3$ such that, for any $i$, there exists $n_i \ge 0$ such that
			\begin{equation*}
				G_i (v)=n_i\left(\dfrac{m_i}{ 2 \pi \, kT}\right)^{3/2} e^{-\frac{m_i}{2kT}|v-u|^2}.
			\end{equation*}
		\end{enumerate}
	\end{enumerate}
\end{proposition}

Choosing $k T=1$, $u=0$, $\nn$ as a nonnegative constant vector, we obtain the normalized centered Maxwell vector function $\MM$ as
\begin{equation*}
	M_i (v)= \left( \dfrac{m_i}{ 2\pi} \right)^{3/2} e^{- \frac{m_i}{2} |v|^2}, \quad 1\le i\le I.
\end{equation*}

Let us then recall the collision invariants in the gas mixtures setting which can be found in \cite{bri-dau}, for instance. The collision invariants are the velocity-depending functions which make the previous weak form of $\QQ$ vanish. They are moreover chosen one-to-one orthogonal and normalized with respect to a $L^2$ scalar product weighted in terms of $\nn\MM$ (remembering the component-wise multiplication defined above). More precisely, we set 
\begin{equation}\label{collision invariants}
\left\{\begin{matrix}
\kki^1 = \frac{1}{\sqrt{n_1}} \begin{pmatrix} 1\\0\\ \vdots \\0 \end{pmatrix}, \quad  
\kki^2 = \frac{1}{\sqrt{n_2}} \begin{pmatrix} 0\\1\\ \vdots \\0 \end{pmatrix}, 
\quad \dots, \quad 
\kki^I = \frac{1}{\sqrt{n_I}} \begin{pmatrix} 0\\0\\ \vdots \\1 \end{pmatrix},\\[3em]
\phantom{\kki}\qquad \kki^{I+1} = \frac{1}{\sqrt{\sum_{j=1}^I n_j m_j}} 
  \begin{pmatrix} m_1 v_1\\m_2 v_1\\ \vdots \\m_I v_1 \end{pmatrix}, \quad \dots, \quad 
\kki^{I+3} = \frac{1}{\sqrt{\sum_{j=1}^I n_j m_j}} 
  \begin{pmatrix} m_1 v_3\\m_2 v_3\\ \vdots \\m_I v_3\end{pmatrix},\\[3em]
\kki^{I+4}=\frac{1}{\sqrt{6 \sum_{j=1}^I n_j}} 
  \begin{pmatrix} m_1 |v|^2 - 3 \\m_2 |v|^2 - 3 \\ \vdots \\m_I|v|^2 - 3 \end{pmatrix}. 
\end{matrix}\right.
\end{equation}
Then the family $(\kki^k)_{1\le k\le I+4}$ satisfies, for any $\GG$,
\begin{align}
& \left\langle \QQ(\GG,\GG), \kki^k \right\rangle_I =0, \qquad 1\le k \le I+4, \nonumber \\[5pt]
& \left\langle \M^{1/2} \kki^k, \M^{1/2} \kki^\ell \right\rangle_I = \delta_{k\ell}, \qquad 1\le k, \ell \le I+4. \label{collision_invariants_orthogonality}
\end{align}

\medskip

In this paper, we focus on a perturbation of the global equilibrium distribution function $\nn \MM$. More precisely, we consider a perturbation carried by a vector-valued function $\ff$, which implies that $\FF$ takes the form
\begin{equation}\label{perturbation}
\FF = \nn \MM + \M^{1/2} \ff.
\end{equation}
Since $\nn\MM$ does not depend on $t$ and $x$, we shall carefully study the macroscopic part of the perturbation, which contains the time and space variations of $\ff$ and subsequently of $\FF$, as in \cite{liu-yu04}, bringing at the same time some more details about the estimates to handle the mixture case. Note that another possibility would have been to follow \cite{liu-yan-yu}, in which $\FF$ is decomposed into the sum of a local Maxwellian, containing the whole macroscopic part of the distribution function, and the microscopic part. This decomposition induces other difficulties, such as the dependence of $\ker\mL$ on~$x$ and $t$.

\section{Main ideas and results} \label{s:mainideas}

Let us now focus our attention on the micro-macro decomposition of the perturbation $\ff$. 
Straightforwardly, \eqref{BEvectorform} implies that $\ff$ satisfies
\begin{equation}\label{equation for f}
\pa_t \ff + v_1 \pa_x \ff  - \mL \ff=\mN (\ff),
\end{equation}
where $\mL$ and $\mN$ are respectively the linearized Boltzmann operator and a quadratic operator defined by
\begin{align*}
\mL \ff &=  \M^{-1/2} \left( \QQ(\nn \MM, \M^{1/2}\ff) + \QQ(\M^{1/2}\ff, \nn \MM) \right), \\
\mN (\ff) &= \M^{-1/2}  \QQ( \M^{1/2}\ff, \M^{1/2}\ff).
\end{align*}
In the remainder of the paper, we shall denote by $\bD$ the domain of $\mL$ in $L^2(\bR^3)^I$.

It is easy to see \cite{bri-dau, nous} that the operator $\mL$ is a non-positive self-adjoint operator, \ie for any $\ff, \gg \in \bD$, 
$$\langle \mL \ff, \gg \rangle_I = \langle \ff, \mL \gg \rangle_I,
\qquad \langle \mL \ff, \ff \rangle_I \le 0.$$
The collisions invariants allow to characterize the elements of $\bP^0=\ker \mL$, \ie
\begin{equation}\label{ker L spanned}
\bP^0=\ker \mL = \Span \left\{ \M^{1/2} \kki^k~|~ 1\le k \le I+4\right\},
\end{equation}
which is a finite-dimensional subspace of $\bD$ with $\dim \bP^0=I+4$. Let us denote by $\bP^1$ its orthogonal complement in $\bD$ with respect to the $\langle\cdot,\cdot\rangle_I$ scalar product, i.e. $\bP^1 = (\ker \mL)^{\perp} = (\bP^0)^{\perp}$.

If we naively proceed by multiplying \eqref{equation for f} by $\ff$ and integrate with respect to $t$, $x$ and $v$, the only term we can hope to upper-bound comes from a spectral gap estimate for $\mL$: it is a square norm of the projection of $\ff$ onto $\bP^1$, which we denote $\ff^1$. That implies that we fail to control the norm of the projection $\ff^0$ of $\ff$ onto $\bP^0$. To control $\ff^0$, we use the so-called micro-macro decomposition of $\ff$, as $\ff$ is uniquely written as 
$$\ff=\ff^0+\ff^1.$$ 
We shall get back to it later, but $\ff^0$ is called the macroscopic part of $\ff$, and $\ff^1$ its microscopic part.

The statements below will be accurately justified in the upcomings sections. The projections of $\ff$ satisfy, as in \cite{liu-yu04},
\begin{align*}
\pa_t \ff^0 + \PP^0(v_1 \pa_x \ff^0) + \PP^0(v_1 \pa_x \ff^1) &=0,\\
\pa_t \ff^1 + \PP^1(v_1 \pa_x \ff^0) + \PP^1(v_1 \pa_x \ff^1) - \mL \ff^1 &=\mN(\ff).
\end{align*}
As we shall see, $\hL=\mL_{|\bP^1}$ is invertible, thus we get, from the previous equation on $\ff^1$,
$$\ff^1= \iL \left( \pa_t \ff^1 + \PP^1(v_1 \pa_x \ff^0) + \PP^1(v_1 \pa_x \ff^1) -\mN(\ff)\right).$$
We now plug this expression of $\ff^1$ in the previous equation on $\ff^0$, which ensures 
$$\pa_t \ff^0 + \PP^0(v_1 \pa_x \ff^0) + \PP^0\left(v_1 \pa_x \iL \left( \pa_t \ff^1 + \PP^1(v_1 \pa_x \ff^0) + \PP^1(v_1 \pa_x \ff^1) -\mN(\ff)\right)\right) =0.$$
We immediately observe that, when scalarily multiplying by $\ff^0$, there is no way to exhibit an estimate on a norm of $\ff^0$ itself. 
There may only be hope to find an estimate of the norm $\|\pa_x\ff^0\|_I^2$. Indeed, we recall that $\bP^0$ is a finite-dimensional space, and $\PP^1(v_1 \pa_x \ff^0)$ can be expressed in terms of the space derivatives of the coordinates of $\ff^0$ in the orthonormal basis of $\bP^0$. Then a subtle combination of arguments, some new and specific to the mixture case, some others coming from \cite{liu-yan-yu,liu-yu04}, leads to an estimate involving $\|\pa_x\ff^0\|_I^2$ on the left-hand side. We still do not have any control on $\|\ff^0\|_I^2$, but such a control is needed, since the nonlinear term involves $\ff$, and not only $\ff^1$ (this is due to the fact that $\mN(\ff^0)\neq 0$), and derivatives of the quadratic term $\mN(\ff)$ involve the derivatives of $\ff$ and $\ff$ itself. We can see that it is possible to do so by introducing the antiderivative $\WW^0$ of $\ff^0$ with respect to $x$, which justifies the fact that we are working in a one-dimensional setting. Without  loss of generality, as we shall explain in Subsection~\ref{ss:hypW0}, we assume that 
\begin{equation}\label{e:hypintf0}
\int_\bR \ff^0(0,x,v)\,\xd x = 0, \qquad v\in\bR^3,
\end{equation}
so that 
$$\WW^0: (t,x,v) \mapsto \int_{-\infty}^x \ff^0(t,y,v)\, \xd y$$
can be treated within a $L^2$-framework with respect to both variables $x$ and $v$. This allows to derive an estimate on $\pa_x\WW^0=\ff^0$, roughly as we did for $\pa_x \ff^0$. Eventually, our global estimate requires the control of norms of the time and space derivatives of $\ff^1$, which can be obtained after differentiation of the equation on $\ff^1$ with respect to $t$ or $x$, thanks to the spectral gap of $\mL$. This whole process also requires a mandatory assumption to deal with the nonlinear term $\mN(\ff)$, which implies that $\WW^0$ and $\ff$ must remain small in the $\|\cdot\|_I$ norm, pointwise in time and space. 

Let us first make the following smallness assumption, where $\sass>0$ will be chosen afterwards. 
\begin{ass}
The perturbation $\ff$ must satisfy
\begin{equation}\label{smallness ass loe}
\sup\limits_{\begin{subarray}{c}t\ge 0\\
		x\in \bR\end{subarray}} 
\left(\|\WW^0\|_I+\|\ff^0\|_I+\|(1+|v|)^{1/2}\ff^1\|_I \right)\le \sass.
\end{equation}
\end{ass}
This assumption is stronger than the one used in \cite{liu-yu04}, where the term $\|\ff^0\|_I^2$ was not involved, but also allows to obtain an estimate on the full macroscopic part, which was not the case previously. More precisely, our first main result is an \textit{a priori} estimate on norms of $\ff$, its partial derivatives, and the antiderivative of $\ff^0$. 
The generic term denoted by $\mI(0)$ which appears below must be understood as a linear combination of square norms of initial data of the pointwise in time integrals of the left-hand side of the estimates, with coefficients only depending on the problem data. 
\begin{proposition}\label{prop:loe}
There exists $\eps_0>0$ such that, for any $\sass\in(0,\eps_0]$, we can find positive constants $\alpha_0$, $\alpha_1$, $\hat C_1$, $\hat C_2$, $\hat C_3$ such that, for any solution $\ff$ to \eqref{equation for f}  satisfying the smallness assumption \eqref{smallness ass loe}, the following lower-order \emph{a priori} estimate holds
\begin{multline}\label{loe}
\frac 14 \int_\bR \|\WW^0\|_I^2\,\xd x \Big|_{t=T}
+\frac 14 \int_\bR \|\ff^0\|_I^2\,\xd x \Big|_{t=T}
+\alpha_0 \int_{\bR} \left\| \ff\right\|^2_I \, \xd x \Big|_{t=T}
+\frac{\alpha_1}2 \int_{\bR} \left\|\pa_x\ff\right\|^2_I \,\xd x\Big|_{t=T} \\
+\frac{\alpha_1}2 \int_{\bR} \left\|\pa_t\ff\right\|^2_I \,\xd x \Big|_{t=T}
+\hat C_1\int_0^T\!\!\!\int_{\bR}\|\ff^0\|_I^2\,\xd x\,\xd t
+\hat C_1\int_0^T\!\!\!\int_{\bR} \left\|\pa_x\ff^0\right\|^2_I \xd x\,\xd t \\
+\hat C_2
\int_0^T\!\!\!\int_{\bR} \left\|(1+|v|)^{1/2}\ff^1\right\|_I^2 \, \xd x \, \xd t 
+2\hat C_3
\int_0^T\!\!\!\int_{\bR} \left\|(1+|v|)^{1/2}\pa_x\ff^1\right\|_I^2 \, \xd x \, \xd t\\
+2\hat C_3
\int_0^T\!\!\!\int_{\bR} \left\|(1+|v|)^{1/2}\pa_t\ff^1\right\|_I^2 \, \xd x \, \xd t
\le \mI(0).
\end{multline}
\end{proposition}
Unfortunately, estimate~\eqref{loe} is not closed. For the latter property to hold, choosing the initial data such that $\mI(0)$ is of order $\sass^2$ would have to imply \eqref{smallness ass loe}. However, $\|(1+|v|)^{1/2}\ff^1\|_I$ would only lie in $H^1_{t,x}$, which does not continuously inject in $L^\infty_{t,x}$. We thus need a higher-order estimate involving more derivatives with respect to $t$ and $x$. In order to obtain it, we differentiate the equations satisfied by $\ff^0$ and $\ff$ as many times as necessary. To this aim, we introduce the following notation. For any $p=(p_1,p_2)\in\bN^2$, we set $|p|=p_1+p_2$, and, for any function $\gg$ of $t$, $x$ and $v$, 
$$\pa^p\gg=\pa_t^{p_1} \pa_x^{p_2} \gg.$$
We emphasize that the multi-index notation does not imply any derivative with respect to $v$. 

At this point, we make the following smallness assumption, where $\sass>0$ will be chosen afterwards. 
\begin{ass}
The perturbation $\ff$ must satisfy
\begin{equation}\label{smallness ass hoe}
\sup\limits_{\begin{subarray}{c}t\geq0\\
		x\in \bR\end{subarray}} 
\left[\|\WW^0\|_I+\max_{|p|\le 2}\left(\|\pa^p\ff^0\|_I+\|(1+|v|)^{1/2}\pa^p\ff^1\|_I \right)\right]\le \sass.
\end{equation}
\end{ass}
Under Assumption~\eqref{smallness ass hoe}, we obtain the next theorem. This time, as we explain below, the smallness assumption can be dropped, provided that we assume instead that $\mI(0)$ is small. 
\begin{theorem}\label{prop higher order estimate}
There exists $\eps_1>0$ such that, for any $\sass\in(0,\eps_1]$, and any solution $\ff$ to \eqref{equation for f} such that the corresponding $\mI(0)$ is at most of order $\sass^2$, the following higher-order \emph{a priori} estimate holds
\begin{multline}\label{e:hoefinal}
\int_\bR \|\WW^0\|_I^2\,\xd x \Big|_{t=T}
+\int_\bR \|\ff^0\|_I^2\,\xd x \Big|_{t=T}
+\sum_{1\le |r|\le 4} \int_\bR \|\pa^r\ff^0\|_I^2\,\xd x \Big|_{t=T} \\
+\int_{\bR} \left\| \ff\right\|^2_I \, \xd x \Big|_{t=T}
+\sum_{1\le|p|\le 5}\int_{\bR} \left\|\pa^p \ff\right\|^2_I \, \xd x \Big|_{t=T} \\
+ \int_0^T\!\!\!\int_{\bR}\|\ff^0\|_I^2\,\xd x\,\xd t
+\int_0^T\!\!\!\int_{\bR} \left\|\pa_x\ff^0\right\|^2_I \xd x\,\xd t
+ \sum_{1\le|r|\le 4}
\int_0^T\!\!\!\int_{\bR} \left\|\pa_x\pa^r\ff^0\right\|^2_I \xd x\,\xd t \\
+\int_0^T\!\!\!\int_{\bR} \left\|(1+|v|)^{1/2}\ff^1\right\|^2_I \xd x\,\xd t
+ \sum_{1\le|p|\le 5}
\int_0^T\!\!\!\int_{\bR} \left\|(1+|v|)^{1/2}\pa^p\ff^1\right\|^2_I \xd x\,\xd t \le \mI(0).
\end{multline}
\end{theorem}
Estimate~\eqref{e:hoefinal} is closed because, this time, $\|(1+|v|)^{1/2}\ff^1\|_I$ lies in $H^5_{t,x} \hookrightarrow W^{2,\infty}_{t,x}$, and $\|\ff^0\|_I$ lies in $H^4_{t,x} \hookrightarrow W^{2,\infty}_{t,x}$. Estimate \eqref{e:hoefinal} implies that Assumption~\eqref{smallness ass hoe} is automatically satisfied when $\mI(0)$ is of order $\sass^2$. The latter closed estimate \eqref{e:hoefinal} appears as a very useful tool. For instance, it allows to obtain the stability of the global Maxwellian function $\M$ in large time, provided that the perturbation at initial time is chosen small enough in the $H^s(L^2_v)$ norm, for $s\ge 5$.

\bigskip

Let us now focus on the proofs of Proposition~\ref{prop:loe} and Theorem~\ref{prop higher order estimate}, starting with the lower-order estimate.

\medskip

\section{Proof of the lower-order estimate} \label{s:loe}

\subsection{Estimates on $\mL$ and $\mN$}
The linearized operator $\mL$ can be written as $\mL=\KK-\nnu$, where $\KK$ is compact \cite{nous} and $\nnu$ is a multiplicative operator, called the collision frequency, given by
$$\nu_i(v) = \sum_{j=1}^I n_j \left( \frac{m_j}{2 \pi} \right)^{3/2} \iint_{\bR^3 \times S^2} e^{-\frac{m_j}{2} |v_*|^2} \mathcal{B}_{ij}(v,v_*,\omega) \, \xd\omega \, \xd v_*, \qquad v\in \bR^3, \quad 1\le i\le I.$$
Note that, thanks to the hard-sphere assumption \eqref{ass Bij L2} on the cross sections, as in \cite{liu-yan-yu}, $\nnu$ satisfies a growth estimate. More precisely, there exist positive constants $\nu_0$ and $\bar\nu_0$, such that 
\begin{equation} \label{e:growthprop}
	0<\nu_0\le\nu_0(1+|v|)\le \nu_i(v) \le \bar\nu_0(1+|v|), \qquad v\in\bR^3, \quad 1\le i\le I.
\end{equation}
Besides, in \cite{bri-dau}, a constructive spectral gap estimate on $\mL$ is proved. In our notation, it means that there exists $\sgd>0$ such that, for any $\hh \in  (\ker \mL)^{\perp}$, 
\begin{equation}\label{spectral gap}
\langle \mL \hh, \hh \rangle_I  \le  - \sgd \| \nnu^{1/2}  \hh \|_I^2.
\end{equation}
The spectral gap estimate \eqref{spectral gap} on $\mL$ and the lower bound $\nu_0$ on $\nnu$ from \eqref{e:growthprop}, together with the Cauchy-Schwarz inequality, yield
$$\nu_0\|\gg\|_I^2 \le \| \nnu^{1/2} \gg \|_I^2 \le -\frac{1}{\sgd} \langle \mL \gg, \gg \rangle_I  \le \frac{1}{\sgd} \| \gg \|_I \| \mL \gg \|_I.$$
Since $\bP^1=(\ker \mL)^{\perp}$ and $\mL$ is self-adjoint, it is clear that $\mL(\bP^1)=\bP^1$, hence $\hL=\mL_{|\bP^1}$ is an invertible operator on $\bP^1$. Now, for any $\hh\in  \bP^1$, writing $\gg=\iL \hh$ and setting $\Cbi = (\nu_0 \sgd)^{-1}>0$, the previous inequality implies that
\begin{equation}\label{Linv bdd}
\| \iL \hh\|_I \le \Cbi \|\hh \|_I,
\end{equation}
which ensures the boundedness of $\iL$ on $\bP^1$. 

Eventually, we write an estimate on the nonlinear operator $\mN$. Its $i$-th component, $1\le i \le I$, is given by
\begin{equation*}
\mathcal N_i (\ff) = M_i^{-1/2} \sum_{j=1}^I n_j^{1/2} Q_{ij}(M_i^{1/2}f_i,M_j^{1/2} f_j).
\end{equation*}
Thanks to Lemma~\ref{prop sulem} stated in Appendix~\ref{appendix-Q}, we immediately have
\begin{align}
\left\|(1+|v|)^{-1/2} \mN(\ff) \right\|_I &\le \Cnmbet 
\left\|(1+|v|)^{1/2} \ff \right\|_I^2, \label{e:estimNf}\\
\left\|(1+|v|)^{-1/2} \pa_\star\mN(\ff) \right\|_I &\le 2\,\Cnmbet 
\left\|(1+|v|)^{1/2} \ff \right\|_I\left\|(1+|v|)^{1/2} \pa_\star\ff \right\|_I,
\label{e:estimDstarNf}
\end{align}
where $\pa_\star$ denotes either time or space partial differentiation.

\subsection{Micro-macro decomposition}

The micro-macro decomposition method lies on the orthogonal decomposition of $\bD$ onto $\bP^0=\ker{\mL}$ and $\bP^1=( \ker \mL )^\perp=\im\mL$. In order to perform it, let $\PP^0$ and $\PP^1$ respectively denote  the orthogonal projections on $\bP^0$ and $\bP^1$. It is clear that
$$\mL \PP^0 = \PP^0\mL = 0, \qquad \mL \PP^1 = \PP^1\mL.$$
Moreover, thanks to \eqref{ker L spanned}, we can write, for any $\gg\in \bD$,
\begin{equation}\label{P0 definition}
\PP^0 \gg = \sum_{k=1}^{I+4} \left\langle \M^{1/2} \kki^k, \gg \right\rangle_I \M^{1/2} \kki^k.
\end{equation}

We decompose the perturbation $\ff$ following the direct orthogonal sum $\bP^0\oplus\bP^1$. Then, $\ff^0=\PP^0\ff$ and $\ff^1=\PP^1\ff$ satisfy
$$\ff= \PP^0\ff+\PP^1\ff= \ff^0+\ff^1, \qquad \langle \ff^0, \ff^1 \rangle_I=0.$$
Functions $\ff^0$ and $\ff^1$ are respectively referred to as the macroscopic (or fluid) and microscopic (or non-fluid) components of $\ff$. 

The coordinates of $\ff^0$ in the orthonormal basis $\left(\M^{1/2} \kki^k \right)_{1\le k \le I+4}$ of $\bP^0$, also known as the fluid quantities, are given by
\begin{equation*}
\begin{split}
\rho_i(t,x) &= \left\langle \M^{1/2} \kki^i, \ff \right\rangle_I, \quad 1\le i\le I, \\
   q^k(t,x) &= \left\langle \M^{1/2} \kki^{I+k}, \ff \right\rangle_I, \quad k=1,2,3, \\
     e(t,x) &= \left\langle \M^{1/2} \kki^{I+4}, \ff \right\rangle_I,
\end{split}
\end{equation*}
so that 
\begin{equation} \label{f0 representation in chi's}
\ff^0 = \sum_{i=1}^I \rho_i   \M^{1/2} \kki^i + \sum_{k=1}^3 q^k  \M^{1/2} \kki^{I+k} + e  \M^{1/2} \kki^{I+4}. 
\end{equation}
Note that~\eqref{f0 representation in chi's} is a very convenient form of $\ff^0$, as it is written as a sum of tensorized functions with respect to $t$ and $x$ on the one hand, and $v$ on the other hand.

If we project the equation \eqref{equation for f} on the perturbation $\ff$ onto $\bP^0$ and $\bP^1$, respectively, we obtain the following equations respectively satisfied by $\ff^0$ and $\ff^1$, \ie
\begin{align}
\pa_t \ff^0 + \PP^0(v_1 \pa_x \ff^0) + \PP^0(v_1 \pa_x \ff^1) 
&=0,\label{equation for f0} \\
\pa_t \ff^1 + \PP^1(v_1 \pa_x \ff^0) + \PP^1(v_1 \pa_x \ff^1) 
- \mL \ff^1 &=\mN(\ff).\label{equation for f1}
\end{align}
The nonlinear term $\PP^0\mN(\ff)$ vanished in \eqref{equation for f0} by combining the representation \eqref{P0 definition} of elements of $\bP^0$, the definition of $\mN$, and the one \eqref{collision_invariants_orthogonality} of the collision invariants. This means that $\mN(\ff)$ lies in $\bP^1$, which we take into account, together with the commutation of $\PP^1$ and $\mL$ to rewrite \eqref{equation for f1} as an equality regarding $\ff^1$, that is
\begin{equation}\label{f1 as L-1}
\ff^1= \iL \left( \pa_t \ff^1 + \PP^1(v_1 \pa_x \ff^0) + \PP^1(v_1 \pa_x \ff^1) -\mN(\ff)\right).
\end{equation}
Of course, \eqref{f1 as L-1} does not provide a direct expression of $\ff^1$, since its right-hand side still depends on the nonlinear term $\mN(\ff)=\mN(\ff^0+\ff^1)$ and first-order derivatives of $\ff^0$ and $\ff^1$. Nevertheless, this equality is crucial for the rest of the paper. Let us also emphasize here that time or space differentiations commute with $\mL$, $\PP^0$ and $\PP^1$, since these three operators only act on the velocity variable. 

Let us now state a result on the fluid quantities, which is important for the further analysis. It comes from the equation \eqref{equation for f0} satisfied by $\ff^0$ and provides conservation laws for $(\rho_i)_{1\le i\le I}$, $(q^k)_{k\in\{1,2,3\}}$ and~$e$.
\begin{proposition}\label{CL prop}
The fluid quantities of $\ff^0$ satisfy the following conservation laws
\begin{align}
&\pa_t \rho_i + \frac{n_i}{\sqrt{\sum n_jm_j}} \pa_x q^1 
+ \left\langle \PP^0(v_1 \pa_x \ff^1), \kki^i \M^{1/2} \right\rangle_I =0, 
\quad 1\le i\le I, \label{CL rhoi}\\
&\pa_t q^1 + \frac{1}{\sqrt{\sum n_jm_j}}  \pa_x 
\left( \sum_i \sqrt{n_i}\rho_i + \frac{\sqrt{6\sum n_j}}{3} e \right)+ 
\left\langle \PP^0(v_1 \pa_x \ff^1), \kki^{I+1} \M^{1/2} \right\rangle_I =0, 
\label{CL q1}\\
&\pa_t q^k + 
\left\langle \PP^0(v_1 \pa_x \ff^1), \kki^{I+k} \M^{1/2} \right\rangle_I =0, 
\qquad k=2,\,3, \label{CL q}\\
&\pa_t e + \frac 13 \sqrt{\frac{6\sum n_j}{\sum n_jm_j}} \pa_x q^1
+ \left\langle \PP^0(v_1 \pa_x \ff^1), \kki^{I+4} \M^{1/2} \right\rangle_I =0
\label{CL e}.
\end{align}
\end{proposition}
The principle of the proof is very simple (checking the equations), but the computations inside are tedious. The proof is provided in Appendix~\ref{s:conslaws} for the sake of completeness. Note that the term with $\ff^1$ in \eqref{CL rhoi} is peculiar to the mixture  and does not appear in the monospecies case. 

In the following, we shall also need to introduce the fluid quantity appearing in \eqref{CL q1} with its space derivative, that is
\begin{equation}\label{ell}
\ell = \frac{1}{ \sqrt{\sum n_j m_j}}   \left( \sum_i \sqrt{n_i}  \rho_i + \frac{2 \sqrt{\sum n_j }}{\sqrt{6}} e \right).
\end{equation}
Using \eqref{CL rhoi} and \eqref{CL e}, $\ell$ clearly satisfies
\begin{multline}\label{e:eqell}
\pa_t \ell + {\frac{\displaystyle\sum {n_j}^{3/2}+\frac 23 \sum n_j}{\displaystyle\sum n_j m_j}} ~\pa_x q^1 \\+\frac 1{\sqrt{\sum n_j m_j}}\left\langle \PP^0(v_1 \pa_x \ff^1), \left(\sum_i \sqrt{n_i}\kki^i+\frac 13 \sqrt{6\sum n_j}\kki^{I+4}\right) \M^{1/2}\right\rangle_I =0,
\end{multline}
so that $q^1$ and $\ell$ can easily be linked through \eqref{CL q1} and \eqref{e:eqell}.

Note that, thanks to \eqref{e:hypintf0}, the perturbation $\ff$ has zero total macroscopic quantities at initial time, \ie
$$\int_\bR \rho_i\,\xd x\Big|_{t=0}=0, \quad 1\le i\le I, \qquad
\int_\bR q^k\,\xd x\Big|_{t=0}=0, \quad 1\le k\le 3, \qquad
\int_\bR e\,\xd x\Big|_{t=0}=0.$$
Integrating the conservation laws \eqref{CL rhoi}--\eqref{CL e} with respect to $x\in\bR$ obviously ensures that
\begin{align}
\int_\bR \rho_i(t,x)\,\xd x&=\int_\bR \rho_i(0,x)\,\xd x, \quad 0\le t\le T,\quad 1\le i\le I, \label{e:intrho}\\
\int_\bR q^k(t,x)\,\xd x&=\int_\bR q^k(0,x)\,\xd x, \quad 0\le t\le T,\quad 1\le k\le 3, \label{e:intq}\\
\int_\bR e(t,x)\,\xd x&=\int_\bR e(0,x)\,\xd x, \quad 0\le t\le T. \label{e:inte}
\end{align}
That allows to define, for any $t$, the antiderivatives $(R_i)$, $(Q^k)$ and $E$ of $(\rho_i)$, $(q^k)$ and $e$ with respect to $x$, which are also the coordinates of $\WW^0$ in the basis $(\kki^k\M^{1/2})$ of $\bP^0$. Denote by $L$ the antiderivative of $\ell$, then we have the straightforward corollary of Proposition~\ref{CL prop}.
\begin{corollary} 
The fluid quantities in $\WW^0$ satisfy
\begin{align}
&\pa_t Q^1 + \ell 
+ \left\langle \PP^0(v_1 \ff^1), \kki^{I+1} \M^{1/2} \right\rangle_I =0, 
\label{CL Q1}\\
&\pa_t L + {\frac{\displaystyle\sum {n_j}^{3/2}+\frac 23 \sum n_j}{\displaystyle\sum n_j m_j}} ~q^1 \label{CL ELL}\\
&\hspace*{1.5cm}+\frac 1{\sqrt{\sum n_j m_j}}\left\langle \PP^0(v_1 \ff^1), \left(\sum_i \sqrt{n_i}\kki^i+\frac 13 \sqrt{6\sum n_j}\kki^{I+4}\right) \M^{1/2}\right\rangle_I =0. \nonumber
\end{align}
\end{corollary}

We conclude this subsection by the following lemma, which is useful in the proofs of the upcoming {\it a priori} estimates. It relies on the fact that $\bP^0$ is finite-dimensional.

\begin{lemma}
The norms $\|\cdot\|_I$, $\|v_1 \cdot\|_I$  and $\|(1+|v|)^{1/2}\cdot\|_I$ are equivalent on $\bP^0$. Hence, there exists a constant $\CkerL>0$, depending on $\nn$ and $\mm$, such that for any $\gg\in \bP^0$, 
\begin{equation}\label{Ceq}
\| v_1 \gg\|_I \le \CkerL \|\gg\|_I \quad \text{ and } \quad \|  (1+|v|)^{1/2}\gg\|_I \le \CkerL \|\gg\|_I. 
\end{equation}
Moreover, from \eqref{P0 definition}, we can deduce that there exists a constant $\Cki>0$ only depending on $\nn$ and $\mm$, such that for any $\ff \in \bD$,
\begin{equation}\label{Cchi}
  \left\|\PP^0(v_1\ff)\right\|_I^2
=\sum_{k=1}^{I+4}
\left| \left\langle v_1\M^{1/2} \kki^k, \ff \right\rangle_I
\right|^2 \le \Cki \| \ff\|_I^2. 
\end{equation}
\end{lemma}

\subsection{Handling the lower-order estimate on $\ff^1$ through $\ff$} \label{ss:estf1}

Since we deal with $\ff^0$ estimates separately, it is equivalent to treat $\ff$ or $\ff^1$ to obtain an estimate on the microscopic part. We choose to proceed with $\ff$. Let us scalarily multiply \eqref{equation for f} by $\ff$, integrate with respect to $t\in [0,T]$ and $x\in\bR$ to obtain 
\begin{equation}\label{e:BEmultf}
\frac 12 \int_{\bR} \left\| \ff\right\|^2_I \, \xd x \Bigg{|}_0^T 
+ \int_0^T\!\!\!\int_{\bR} \langle v_1 \pa_x \ff,\ff \rangle_I \, \xd x \, \xd t
- \int_0^T\!\!\!\int_{\bR} \langle \mL \ff, \ff \rangle_I \, \xd x \, \xd t  \le  \int_0^T\!\!\!\int_{\bR} \langle \mN(\ff), \ff\rangle_I \, \xd x \, \xd t.
\end{equation}
The second term on the left-hand side of this equation vanishes by conservativity with respect to the $x$ variable, since 
$$\int_0^T\!\!\!\int_{\bR} \langle v_1 \pa_x \ff,\ff \rangle_I \, \xd x \, \xd t= \frac 12 \int_0^T\!\!\!\int_{\bR} \frac\pa{\pa x} \left( \sum_{i=1}^I\int_{\bR^3} v_1 f_i^2\,\xd v \right)\,\xd x\, \xd t.$$
The term with $\mL$ can be lower-bounded thanks to the $\mL$ spectral gap estimate \eqref{spectral gap} and to \eqref{e:growthprop} in the following way
$$\sgd \nu_0 \|(1+|v|)^{1/2}\ff^1\|_I^2 \le \sgd \|\nnu^{1/2}\ff^1\|_I^2 \le -\langle \mL \ff^1, \ff^1 \rangle_I = -\langle \mL \ff, \ff \rangle_I.$$
Eventually, we deal with the term involving $\mN$. We first notice that 
$$\langle \mN(\ff),\ff\rangle_I = \langle \mN(\ff),\ff^1\rangle_I=
\langle (1+|v|)^{-1/2}\mN(\ff),(1+|v|)^{1/2}\ff^1\rangle_I,$$
since $\PP^0\mN(\ff)=0$. Then, using \eqref{e:estimNf}, we have
\begin{multline} \label{e:Ntrick}
\left|\langle \mN(\ff),\ff\rangle_I\right| \le 
\Cnmbet\left\|(1+|v|)^{1/2}\ff\right\|_I^2 \left\|(1+|v|)^{1/2}\ff^1\right\|_I \\
\le 2\Cnmbet\left( \left\|(1+|v|)^{1/2}\ff^0\right\|_I^2 
+ \left\|(1+|v|)^{1/2}\ff^1\right\|_I^2\right)
\left\|(1+|v|)^{1/2}\ff^1\right\|_I.
\end{multline}
Using the equivalence of norms \eqref{Ceq} on $\bP^0$, we can find a constant $\KkerL>0$, depending on $\mm$, $\nn$, the cross sections and $\mL$ through its null space, such that \eqref{e:BEmultf} becomes
\begin{multline}\label{e:estfprf1}
\frac 12 \int_{\bR} \left\| \ff\right\|^2_I \, \xd x \Bigg{|}_0^T 
+\sgd \nu_0 \int_0^T\!\!\!\int_{\bR} \|(1+|v|)^{1/2}\ff^1\|_I^2 \, \xd x \, \xd t \\
\le \KkerL \int_0^T\!\!\!\int_{\bR} \left(\left\|\ff^0\right\|_I^2+\left\|(1+|v|)^{1/2}\ff^1\right\|_I^2\right)\left\|(1+|v|)^{1/2}\ff^1\right\|_I\, \xd x \, \xd t.
\end{multline}
Of course, we can see that, if we choose $\sass$ small enough in assumption \eqref{smallness ass loe}, there will only remain a nonnegative contribution of a norm of $\ff^0$ on the right-hand side of \eqref{e:estfprf1}. This is why we have to focus now on an estimate on $\ff^0$.

\subsection{Lower-order estimate on $\ff^0$: first steps} \label{ss:1ststepf0}
In order to deal with the estimate on $\ff^0$, we plug the expression \eqref{f1 as L-1} of $\ff^1$ into the equation \eqref{equation for f0} satisfied by $\ff^0$, then scalarily multiply the new equation by $\ff^0$, and integrate with respect to $t$ and $x$, to obtain
\begin{equation}\label{e:eqf0avecL1-4}
 \frac12 \int_\bR \|\ff^0\|_I^2\,\xd x \Big|_0^T + L_1 + L_2 + L_3 + L_4 =0,
\end{equation}
where the term with $\langle \PP^0(v_1 \pa_x \ff^0),\ff^0\rangle_I$ vanishes, again by conservativity in $x$, and where we set 
\begin{align*}
L_1 &= \int_{0}^T\!\!\! \int_{\bR} \left\langle v_1 \pa_x \iL \pa_t \ff^1 ,  \ff^0\right\rangle_I \xd x \, \xd t
=-\int_{0}^T\!\!\! \int_{\bR} \left\langle \iL \pa_t\ff^1,
\PP^1(v_1 \pa_x \ff^0)\right\rangle_I \xd x \, \xd t,\\
L_2 &= \int_{0}^T\!\!\! \int_{\bR} \left\langle v_1 \pa_x \iL \PP^1(v_1 \pa_x \ff^0),   \ff^0\right\rangle_I \xd x \, \xd t
=-\int_{0}^T\!\!\! \int_{\bR} \left\langle \iL \PP^1(v_1 \pa_x \ff^0), \PP^1(v_1 \pa_x  \ff^0) \right\rangle_I \, \xd x \, \xd t,\\
L_3 &= \int_{0}^T\!\!\! \int_{\bR} \left\langle v_1 \pa_x \iL \PP^1(v_1 \pa_x \ff^1) ,  \ff^0\right\rangle_I \xd x \, \xd t
=-\int_{0}^T\!\!\! \int_{\bR} \left\langle \pa_x \ff^1, v_1\iL\PP^1(v_1 \pa_x \ff^0) \right\rangle_I \xd x \, \xd t,\\
L_4 &=-\int_{0}^T\!\!\! \int_{\bR} \left\langle v_1 \pa_x \iL \mN(\ff),\ff^0\right\rangle_I \xd x \, \xd t
=\int_{0}^T\!\!\! \int_{\bR} \left\langle \mN(\ff),\iL \PP^1(v_1 \pa_x  \ff^0)\right\rangle_I \xd x \, \xd t.
\end{align*}
Let us show how to handle these terms by performing the following preliminary computations on $L_1$, $L_2$, $L_3$, $L_4$.

\subsubsection{Term $L_1$}  \label{sss:L1}
Combining the equivalence of norms \eqref{Ceq} on $\bP^0$ with the boundedness of $\iL$, and setting $K_1=\Cbi\CkerL>0$, we obtain
\begin{equation}\label{e:esttmpL1}
|L_1| \le K_1\int_{0}^T\!\!\! \int_{\bR} 
\left\| \pa_t \ff^1 \right\|_I \left\|\pa_x\ff^0\right\|_I\,\xd x \, \xd t.
\end{equation}

\subsubsection{Term $L_2$} \label{sss:L2}
It is clear that $V=\PP^1(v_1 \bP^0)$ is a finite-dimensional subspace of $\bP^1$. Consequently, $\iL_{|V}$ is a bounded invertible operator, as well as its inverse. Hence, there exists a constant $\CbdV>0$, only depending on $\mL$, such that, for any $\gg\in\bP^0$, 
\begin{equation*}
\left\| \PP^1(v_1 \gg) \right\|_I 
\le \CbdV \left\| \iL \PP^1(v_1 \gg) \right\|_I.
\end{equation*}
Thus, thanks to the previous inequality and to the spectral gap estimate \eqref{spectral gap} on $\mL$, we get
\begin{align}
L_2&= 
-\int_0^T\!\!\!\int_{\bR} 
\left\langle  \iL \PP^1(v_1 \pa_x \ff^0), \mL\iL \PP^1 (v_1 \pa_x \ff^0)\right\rangle_I \xd x \,\xd t \nonumber\\
&\ge \sgd \int_0^T\!\!\!\int_{\bR} \left\|\nnu^{1/2} \iL \PP^1(v_1 \pa_x \ff^0) \right\|^2_I \xd x\, \xd t \nonumber\\
&\ge \lambda_2\int_0^T\!\!\!\int_{\bR} \left\|\PP^1(v_1 \pa_x \ff^0) \right\|^2_I \xd x\, \xd t, \label{e:esttmpL2}
\end{align}
where we chose $\lambda_2=\sgd \nu_0/\CbdV>0$.

\subsubsection{Term $L_3$} \label{sss:L3}
We apply the Cauchy-Schwarz inequality to yield
$$|L_3| \le \int_{0}^T\!\!\!\int_{\bR} \left\|\pa_x \ff^1\right\|_I 
\left\| v_1 \iL \PP^1(v_1 \pa_x \ff^0) \right\|_I  \xd x \, \xd t.$$
Then, introducing $C_{\KK}>0$ as the boundedness constant of the compact operator 
$\KK=\mL+\nnu$, we notice that
\begin{multline*}
\left\| v_1 \iL \PP^1(v_1 \pa_x \ff^0) \right\|_I
\le \frac 1{\nu_0} \left\|\nnu \iL \PP^1(v_1 \pa_x \ff^0) \right\|_I \\
\le \frac 1{\nu_0} \left(\left\|\KK \iL \PP^1(v_1 \pa_x \ff^0) \right\|_I
+ \left\|\mL \iL \PP^1(v_1 \pa_x \ff^0) \right\|_I \right)
\le \frac\CkerL{\nu_0}(C_{\KK}\Cbi+1) \|\pa_x\ff^0\|_I,
\end{multline*}
where we also used the boundedness of $\iL$ and $\PP^1$ (as a projector), and the norm equivalence argument on $\bP^0$, involving the constant $\CkerL$. Setting $K_3=\CkerL(C_{\KK}\Cbi+1)/\nu_0>0$, we obtain
\begin{equation}\label{e:esttmpL3}
|L_3| \le K_3 \int_0^T\!\!\!\int_{\bR} 
\left\|\pa_x\ff^0\right\|_I \left\|\pa_x \ff^1\right\|_I\,\xd x\,\xd t,
\end{equation}

\subsubsection{Term $L_4$} \label{sss:L4}
We first have to estimate
\begin{align*}
\left\langle \mN(\ff),\iL \PP^1(v_1 \pa_x  \ff^0)\right\rangle_I 
&= \left\langle (1+|v|)^{-1/2}\mN(\ff),(1+|v|)^{1/2}\iL \PP^1(v_1 \pa_x  \ff^0)\right\rangle_I \\
&\le \left\|(1+|v|)^{-1/2}\mN(\ff) \right\|_I \left\|(1+|v|)^{1/2}\iL \PP^1(v_1 \pa_x  \ff^0)\right\|_I.
\end{align*}
We treat the norm with $\mN$ using \eqref{e:estimNf} and the same kind of argument as in \eqref{e:Ntrick}, including norm equivalence on $\bP^0$. It ensures that
$$\left\|(1+|v|)^{-1/2}\mN(\ff) \right\|_I \le
\CkerL \left(\|\ff^0\|_I^2+\|(1+|v|)^{1/2}\ff^1\|_I^2\right).$$
Besides, the norm with $\iL$ is treated in the same way as for $L_3$, \ie
$$\left\|(1+|v|)^{1/2}\iL \PP^1(v_1 \pa_x  \ff^0)\right\|_I \le
\left\|(1+|v|)\iL \PP^1(v_1 \pa_x  \ff^0)\right\|_I 
\le K_3 \|\pa_x\ff^0\|_I.$$
Therefore, setting $K_4={\CkerL}K_3>0$, we can write
\begin{equation}\label{e:esttmpL4}
|L_4| \le K_4\int_0^T\!\!\!\int_{\bR} 
\left(\|\ff^0\|_I^2+\|(1+|v|)^{1/2}\ff^1\|_I^2\right) \|\pa_x\ff^0\|_I
\xd x \, \xd t. 
\end{equation}

\subsubsection{Before we proceed}
Let us explain what the current situation on $\ff^0$ is. 
Applying Young's inequality in \eqref{e:esttmpL1} and \eqref{e:esttmpL3} with a parameter $\delta_0>0$ to be chosen later, and
plugging \eqref{e:esttmpL1}--\eqref{e:esttmpL4} in \eqref{e:eqf0avecL1-4},
we get
\begin{multline} \label{e:bilanestf0}
\frac12 \int_{\bR} \|\ff^0\|_I^2\,\xd x \Big|_0^T + 
\lambda_2\int_0^T\!\!\!\int_{\bR} \left\|\PP^1(v_1\pa_x\ff^0)\right\|^2_I \xd x\,\xd t\\
\le K_1 \int_0^T\!\!\!\int_{\bR}
\left(\frac1{2\delta_0}\left\|\pa_t\ff^1\right\|_I^2 + \frac{\delta_0}2\left\|\pa_x\ff^0\right\|_I^2\right) \xd x\,\xd t
+ K_3 \int_0^T\!\!\!\int_{\bR}
\left(\frac{\delta_0}2\left\|\pa_x\ff^0\right\|_I^2 + \frac1{2\delta_0}\left\|\pa_x\ff^1\right\|_I^2\right) \xd x\,\xd t \\
+K_4 \int_0^T\!\!\!\int_{\bR}\left\|\pa_x\ff^0\right\|_I
\left(\left\|\ff^0\right\|_I^2+\left\|(1+|v|)^{1/2}\ff^1\right\|_I^2\right)
\xd x \,\xd t.
\end{multline}

We intend afterwards to combine \eqref{e:bilanestf0} with the estimate \eqref{e:estfprf1} for $\ff^1$. All the terms are product of exactly two norms, apart from the ones coming from the nonlinear operator. This is why, in order to proceed, we shall need the same kind of smallness assumption as in \cite{liu-yu04}, namely \eqref{smallness ass loe}. 

The term $\|(1+|v|)^{1/2}\ff^1\|_I^2$ appears in the left-hand side of \eqref{e:estfprf1}, and in the right-hand side of both estimates in the terms coming from $\mN$,  with a multiplication by the arbitrary small parameter $\sass$ when using the smallness assumption. So they can be put on the left-hand side  to obtain a still positive coefficient for $\|(1+|v|)^{1/2}\ff^1\|_I^2$. On the contrary, the time and space derivatives of $\ff^1$ only appear in the right-hand sides of our estimates, and so does $\|\ff^0\|_I^2$. We have no way to control them for the time being. Last, the space derivative of $\ff^0$ appears on the right-hand side of \eqref{e:bilanestf0} as its $I$-norm, and on the left-hand side in $\|\PP^1(v_1 \pa_x \ff^0)\|_I$, which will provide a helpful contribution. 

Consequently, our next two steps are natural: deal with the norms of the time and space derivatives of $\ff^1$, and $\|\PP^1(v_1 \pa_x \ff^0)\|_I$. 

\subsection{Handling the derivatives of $\ff^1$}\label{ss:estDstarf1}
We proceed in the same way as in Subsection~\ref{ss:estf1}. We denote by $\pa_\star$ any time or space partial differentiation. We differentiate \eqref{equation for f} accordingly, scalarily multiply it by $\pa_\star \ff$ and integrate with respect to $t$ and $x$. Then, using the spectral gap \eqref{spectral gap} of $\mL$, the growth property \eqref{e:growthprop} of $\nnu$ and the estimate \eqref{e:estimDstarNf} on $\pa_\star \mN(\ff)$, we obtain, similarly to \eqref{e:estfprf1}, the existence of $C_\star>0$, depending on $\mm$, $\nn$ and $\beta$ and $\mL$ through its null space, such that 
\begin{multline}\label{e:estDstarfprDstarf1-begin}
\frac 12 \int_{\bR} \left\|\pa_\star \ff\right\|^2_I \, \xd x \Bigg{|}_0^T +
\sgd\nu_0 \int_0^T\!\!\!\int_{\bR} \|(1+|v|)^{1/2}\pa_\star\ff^1\|_I^2 \, \xd x \, \xd t \\
\le C_\star \int_0^T\!\!\!\int_{\bR} \left\|(1+|v|)^{1/2}\ff\right\|_I
\left\|(1+|v|)^{1/2}\pa_\star\ff^1\right\|_I \left( 
\left\|\pa_\star\ff^0\right\|_I +\left\|(1+|v|)^{1/2}\pa_\star\ff^1\right\|_I\right)\, \xd x \, \xd t.
\end{multline}
Let us deal with the time derivative, since no other term involving $\pa_t \ff^0$ will appear in our estimates. In fact, using  \eqref{equation for f0} and \eqref{Cchi}, we get, using again the Cauchy-Schwarz inequality, 
$$\left\|\pa_t\ff^0\right\|_I^2 = \left\|\PP^0(v_1\pa_x\ff)\right\|_I^2
\le \Cki \|\pa_x \ff\|_I^2. $$
Consequently, \eqref{e:estDstarfprDstarf1-begin} becomes, for the time derivative, 
\begin{multline}\label{e:estDtfprDtf1}
\frac 12 \int_{\bR} \left\|\pa_t \ff\right\|^2_I \, \xd x \Bigg{|}_0^T +
\sgd\nu_0 
\int_0^T \int_{\bR} \|(1+|v|)^{1/2}\pa_t\ff^1\|_I^2 \, \xd x \, \xd t \\
\le C_\star' \int_0^T\!\!\!\int_{\bR} \left\|(1+|v|)^{1/2}\ff\right\|_I
\left\|(1+|v|)^{1/2}\pa_t\ff^1\right\|_I \left( \left\|\pa_x \ff^0\right\|_I 
+\left\|(1+|v|)^{1/2}\pa_t\ff^1\right\|_I\right)\, \xd x \, \xd t,
\end{multline}
where $C_\star'=C_\star \max(\Cki,1)$.

We can thus observe that \eqref{e:estDstarfprDstarf1-begin} and \eqref{e:estDtfprDtf1} have exactly the same structure and can be rewritten, for some constant $K_5>0$ only depending on $\nn$, $\mm$ and $\beta$, as
\begin{multline}\label{e:estDstarfprDstarf1}
\frac 12 \int_{\bR} \left\|\pa_\star \ff\right\|^2_I \, \xd x \Bigg{|}_0^T +
C_0
\int_0^T\!\!\!\int_{\bR} \|(1+|v|)^{1/2}\pa_\star\ff^1\|_I^2 \, \xd x \, \xd t \\
\le K_5 \int_0^T\!\!\!\int_{\bR} \left\|(1+|v|)^{1/2}\ff\right\|_I
\left\|(1+|v|)^{1/2}\pa_\star\ff^1\right\|_I \left( 
\left\|\pa_x\ff^0\right\|_I +\left\|(1+|v|)^{1/2}\pa_\star\ff^1\right\|_I\right)\, \xd x \, \xd t,
\end{multline}
where we set $C_0=\sgd \nu_0>0$.

\subsection{Lower bound for $\|\PP^1(v_1 \pa_x \ff^0)\|_I$} \label{ss:P1Dxf0}
Let us now focus on the term with $\|\PP^1(v_1 \pa_x \ff^0)\|_I$. The following lemma allows to estimate the $L^2$-norm of $\PP^1(v_1 \pa_x \ff^0)$ in terms of $\pa_x \ff^0$, up to a contribution in $\pa_x \ff^1$ which can be as small as desired. We must emphasize that this result is the main improvement in the mixture case, compared to its monospecies counterparts from \cite[estimate~(2.22) p.185]{liu-yan-yu} and \cite[Lemma~3.1 p.139]{liu-yu04}. Indeed, in the monospecies case, the lower bound of the term $\|\PP^1(v_1 \pa_x \ff^0)\|_I$ did not give any $\pa_x\rho$ contribution, only $\pa_x q$ and $\pa_xe$ parts of $\|\pa_x \ff^0\|_I$. 
\begin{lemma}\label{Lemma Yless 3.3} 
There exists $\theta_0\in(0,1]$ such that, for any $\ff \in \bD$ and any $0<\theta<\theta_0$, 
\begin{multline}\label{est P1 v1 dx f0}
\int_0^T\!\!\!\int_{\bR}\| \PP^1(v_1 \pa_x \ff^0)\|_I^2\,\xd x\,\xd t \ge 
\Gnm \int_0^T\!\!\!\int_{\bR} \|\pa_x \ff^0\|_I^2\,\xd x\,\xd t \\
- \theta C_2 \int_0^T\!\!\!\int_{\bR} \|\pa_x \ff^1\|_I^2\,\xd x\,\xd t
+2\theta \int_{\bR} q\,\pa_x \ell\,\xd x\Big{|}_0^T
\end{multline}
where $\Gnm>0$ only depends on $\nn$, $\mm$ and $\theta$, $C_2>0$ on $\nn$ and $\mm$.
\end{lemma}

\begin{proof} 
By orthogonality, we first have
$$\| \PP^1(v_1 \pa_x \ff^0)\|_I^2 = \| v_1 \pa_x \ff^0\|_I^2 - \|\PP^0(v_1 \pa_x \ff^0)\|_I^2.$$
The first term of the right-hand side writes
\begin{multline*}
\|v_1 \pa_x \ff^0\|_I^2 = \sum_i \frac{1}{m_i} (\pa_x \rho_i)^2 
+ \frac{4}{\sqrt{6 \sum n_j}} \sum_i \frac{\sqrt{n_i}}{m_i} (\pa_x \rho_i)(\pa_x e)\\
+ \frac{7}{3} \frac{1}{\sum n_j} \sum_i \frac{n_i}{m_i}(\pa_x e)^2
+\frac{\sum n_j}{\sum n_j m_j} \left[3\left(\pa_x q^1\right)^2+\left(\pa_x q^2\right)^2 +\left(\pa_x q^3\right)^2 \right].
	\end{multline*}
It can become a sum of square quantities, as
\begin{multline} \label{e:v1dxf0norm}
\|v_1 \pa_x \ff^0\|_I^2 
= \sum_i \frac{1}{m_i} 
\left(\pa_x \rho_i + \frac{2\sqrt{n_i}}{\sqrt{6 \sum n_j}} \pa_x e \right)^2
+ \frac{5}{3} \frac{1}{\sum n_j} \sum_i \frac{n_i}{m_i}(\pa_x e)^2 \\
+\frac{\sum n_j}{\sum n_j m_j} \left[3\left(\pa_x q^1\right)^2+\left(\pa_x q^2\right)^2 +\left(\pa_x q^3\right)^2 \right].
\end{multline}	
The term with $\PP^0$ decomposes on our basis of $\bP^0$ into
\begin{multline*}
\PP^0 (v_1\pa_x \ff^0) = \pa_x q^1 \sum_i \frac{\sqrt{n_i}}{\sqrt{\sum n_jm_j}} \kki^i \M^{1/2} + \pa_x q^1\frac{2\sqrt{\sum n_j}}{\sqrt{6\sum n_jm_j}}  \kki^{I+4} \M^{1/2} \\
+ \frac1{\sqrt{\sum n_jm_j}}
\left(\sum_i \sqrt{n_i} \pa_x \rho_i  + \frac {2\sqrt{\sum n_j}}{\sqrt{6}} \pa_x e\right) \kki^{I+1} \M^{1/2}.
\end{multline*}
It is thus easy to check that
\begin{equation}\label{P0 v grad f0}
\| \PP^0 (v_1\pa_x \ff^0) \|_I^2 = (\pa_x \ell)^2 
+ \frac 53\frac{\sum n_j}{\sum n_jm_j} (\pa_x q^1)^2.
\end{equation}
Therefore, we immediately get
\begin{multline}\label{eq:before_CS}
\|\PP^1(v_1 \pa_x \ff^0)\|_I ^2	=
\sum_i \frac{1}{m_i} \left(\pa_x \rho_i + \frac{2\sqrt{n_i}}{\sqrt{6\sum n_j}} \pa_x e \right)^2 
+\frac{5}{3} \frac{1}{\sum n_j} \left( \sum \frac{n_j}{m_j}\right) (\pa_x e)^2 \\
+ \frac{\sum n_j}{\sum n_j m_j} \left[ \frac{4}{3} \left( \pa_x  q^1 \right)^2 +  \left(\pa_x q^2\right)^2 +  \left(\pa_x q^3\right)^2 \right] - (\pa_x \ell)^2.
\end{multline}
Thanks to the Cauchy-Schwarz inequality, we have
\begin{multline*}
\left(\sum_i \sqrt{n_i} \pa_x \rho_i + \frac{2\sqrt{\sum n_j}}{\sqrt 6} \pa_x e  \right)^2 = 
\left[ \sum_i \left( \sqrt{n_i} \pa_x\rho_i + \frac{2n_i}{\sqrt{6\sum n_j}} \pa_x e\right)\right]^2 \\
 \le \left( \sum_j n_j m_j \right)
\sum_i \frac{1}{m_i} \left(\pa_x \rho_i + \frac{2\sqrt{n_i}}{\sqrt{6\sum n_j}} \pa_x e \right)^2,
\end{multline*}
which implies that
\[(\pa_x \ell)^2\leq \sum_i \frac{1}{m_i} \left(\pa_x \rho_i + \frac{2\sqrt{n_i}}{\sqrt{6\sum n_j}} \pa_x e \right)^2.\]
Let $\theta\in(0,1)$. Then we split $(\pa_x \ell)^2$ in \eqref{eq:before_CS} into the sums of itself respectively multiplied by $\theta$ and $(1-\theta)$. Applying the previous Cauchy-Schwarz inequality, it follows from \eqref{eq:before_CS} that 	
\begin{multline} \label{est P1 v1 dx f0 with theta}
\|\PP^1(v_1 \pa_x \ff^0)\|_I ^2	\ge
\theta\sum_i \frac{1}{m_i} \left(\pa_x \rho_i + \frac{2\sqrt{n_i}}{\sqrt{6\sum n_j}} \pa_x e \right)^2 
+\frac{5}{3} \frac{1}{\sum n_j} \left( \sum \frac{n_j}{m_j}\right) (\pa_x e)^2 \\
+ \frac{\sum n_j}{\sum n_j m_j} \left[ \frac{4}{3} \left( \pa_x  q^1 \right)^2 +  \left(\pa_x q^2\right)^2 +  \left(\pa_x q^3\right)^2 \right]
- \theta (\pa_x \ell)^2.
\end{multline}
The term $-\theta  |\pa_x \ell|^2$ appearing in \eqref{est P1 v1 dx f0 with theta} is absorbed thanks to the conservation laws. More precisely, multiplying \eqref{CL q1} by $\pa_x \ell$ and integrating with respect to $t$ and $x$ gives
\begin{equation*}
\int_0^T\!\!\!\int_{\bR}|\pa_x \ell|^2\,\xd x\,\xd t =
-\int_0^T\!\!\!\int_{\bR} \pa_t q^1 \pa_x \ell\,\xd x\,\xd t
-\int_0^T\!\!\!\int_{\bR} \left\langle \pa_x \ff^1, v_1\kki^{I+1} \M^{1/2}\right\rangle_I \pa_x \ell.
 \end{equation*}
Then, using integrations by parts, one with respect to $t$, one for $x$, for the first term on the right-hand side, and Cauchy-Schwarz and Young's inequalities, we obtain with \eqref{Cchi} that 
\begin{multline*}
\int_0^T\!\!\!\int_{\bR}|\pa_x \ell|^2\,\xd x\,\xd t\le
-\int_0^T\!\!\!\int_{\bR} \pa_t \ell \pa_x q^1\,\xd x\,\xd t - \int_{\bR} q^1 \pa_x \ell\,\xd x \Big|_0^T \\
+ \frac12 \int_0^T\!\!\!\int_{\bR}|\pa_x \ell|^2\,\xd x\,\xd t 
+ \frac{\Cki}2 \int_0^T\!\!\!\int_{\bR} \|\pa_x \ff^1\|_I^2\,\xd x\,\xd t,
\end{multline*}
ensuring that
\begin{equation}\label{eq dx ell}
\frac12 \int_0^T\!\!\!\int_{\bR}|\pa_x \ell|^2\,\xd x\,\xd t 
+ \int_{\bR} q^1 \pa_x \ell\,\xd x \Big|_0^T \le
\left| \int_0^T\!\!\!\int_{\bR} \pa_t \ell\, \pa_x q^1\,\xd x\,\xd t\right| 
+\frac{\Cki}2 \int_0^T\!\!\!\int_{\bR} \|\pa_x \ff^1\|_I^2\,\xd x\,\xd t.
\end{equation}
Further, we multiply \eqref{e:eqell} by $\pa_x q^1$, integrate with respect to $x$ and $t$ and use again Cauchy-Schwarz and Young's inequalities as well as \eqref{Cchi}, to get that there exists a some  constant $K_{\chi}>0$ depending only on $\nn$ and $\mm$ such that
\begin{multline*}
\int_0^T\!\!\!\int_{\bR} \pa_t \ell\, \pa_x q^1\,\xd x\,\xd t =
\frac1{\sum n_j m_j}\left(\sum n_j^{3/2} + \frac23 \sum n_j\right)
\int_0^T\!\!\!\int_{\bR}|\pa_x q^1|^2\,\xd x\,\xd t\\
+\frac12 \int_0^T\!\!\!\int_{\bR}|\pa_x q^1|^2\,\xd x\,\xd t + \frac{K_{\chi}}2  \int_0^T\!\!\!\int_{\bR} \|\pa_x \ff^1\|_I^2\,\xd x\,\xd t.
\end{multline*}
This estimate, combined with \eqref{eq dx ell}, gives, after multiplication by $2\theta$,
\begin{multline}\label{theta dx ell}
\theta \int_0^T\!\!\!\int_{\bR}|\pa_x \ell|^2\,\xd x\,\xd t
+ 2\theta\int_{\bR} q^1 \pa_x \ell\,\xd x \Big|_0^T \\
\le 2\theta \left(\frac{\sum n_j^{3/2} + \frac23 \sum n_j}{\sum n_j m_j} + \frac12 \right)\int_0^T\!\!\!\int_{\bR}|\pa_x q^1|^2\,\xd x\,\xd t 
+ \theta (\Cki + K_{\chi}) \int_0^T\!\!\!\int_{\bR} \|\pa_x \ff^1\|_I^2\,\xd x\,\xd t.
\end{multline}
Setting
\begin{equation}\label{theta0}
\theta_0 = \min\left\{1,\frac{\sum n_j}{{3\sum n_j m_j+6\sum n_j^{3/2} + 4 \sum n_j}}\right\}\in (0,1],
\end{equation}
we can come back to \eqref{est P1 v1 dx f0 with theta} with a fixed $0<\theta<\theta_0$, so that the term in $\|\partial_x q^1\|^2$ in \eqref{theta dx ell} can be absorbed by the following term, taken in \eqref{est P1 v1 dx f0 with theta},
$$\frac{\sum n_j}{\sum n_j m_j} \int_0^T\!\!\!\int_{\bR}\frac 13(\pa_x q^1)^2\,\xd x\,\xd t.$$

Let us now introduce $\mathcal G :\bP^0\to \bR_+$, 
\begin{multline*}
\ff^0 \mapsto 
\theta\sum_i \frac{1}{m_i} 
\left(\rho_i + \frac{2\sqrt{n_i}}{\sqrt{6\sum n_j}} e \right)^2 
+\frac{5}{3} \frac{1}{\sum n_j} \left( \sum \frac{n_j}{m_j}\right) e^2 
+ \frac{\sum n_j}{\sum n_j m_j} \left[ \left(q^1 \right)^2 + \left(q^2\right)^2 +  \left(q^3\right)^2 \right],
\end{multline*}
which defines a positive definite quadratic form on the finite-dimensional subspace $\bP^0$, thus equivalent to $\|\cdot\|_I^2$ on $\bP^0$. Therefore, there exists a constant $\Gnm>0$, only depending on $\nn$, $\mm$ and $\theta$, such that, for any $\ff^0\in\bP^0$,
$$\mathcal G(\ff^0)\ge\Gnm \|\ff^0\|_I^2.$$
All the previous considerations lead to the following estimate
\begin{multline} \label{e:lowestintP1vDxf0}
\int_0^T\int_{\bR}\|\PP^1(v_1\pa_x \ff^0)\|_I^2\,\xd x\,\xd t
\ge \Gnm \int_0^T\!\!\!\int_{\bR}\|\pa_x\ff^0\|_I^2\,\xd x\,\xd t \\
- \theta (\Cki + K_{\chi})   \int_0^T\!\!\!\int_{\bR} \|\pa_x \ff^1\|_I^2\,\xd x\,\xd t
+ 2\theta\int_{\bR} q^1 \partial_x \ell \Big|_0^T.
\end{multline}
We obtain the required estimate \eqref{est P1 v1 dx f0} by setting $C_2 = (\Cki + K_{\chi}) $.
\end{proof}

\begin{remark}
Let us emphasize that the previous lemma holds in higher dimensions. We explain in Appendix~\ref{s:korn} how to handle the proof of Lemma~\ref{Lemma Yless 3.3} in a three-dimensional setting.
\end{remark}

\subsection{Estimate using an antiderivative of $\ff^0$} \label{ss:estW0}
As we already pointed out at the end of Subsection~\ref{ss:1ststepf0}, we still fail to control $\|\ff^0\|_I$. Since we work in a one-dimensional space setting, following the strategy of \cite{liu-yu04}, we write an estimate on the antiderivative $\WW^0$ of $\ff^0$.

We first integrate \eqref{equation for f0} with respect to the space variable between $-\infty$ and $x\in\bR$, so that
\begin{equation}\label{eq W0}
\pa_t\WW^0 + \PP^0\left(v_1 \pa_x\WW^0 \right) + \PP^0\left(v_1\ff^1\right)=0.
\end{equation}
We then scalarily multiply \eqref{eq W0} by $\WW^0$ and integrate with respect to $t$ and $x$. Since the second term has a conservative form in $x$ and using the expression \eqref{f1 as L-1} of $\ff^1$, we get
\begin{equation} \label{e:debutestW0}
\frac12 \int_\bR \|\WW^0\|_I^2\,\xd x \Big|_0^T 
+ \hat L_1+ \hat L_2+ \hat L_3+ \hat L_4=0,
\end{equation}
where we set 
\begin{align*}
\hat L_1 &= \int_0^T\!\!\!\int_{\bR} 
\left\langle\PP^0(v_1 \iL\pa_t\ff^1),\WW^0\right\rangle_I\,\xd x\,\xd t \\
\hat L_2 &= - \int_0^T\!\!\!\int_{\bR} 
\left\langle \iL \PP^1(v_1\ff^0),\PP^1(v_1\ff^0)\right\rangle_I\,\xd x\,\xd t \\
\hat L_3 &= - \int_0^T\!\!\!\int_{\bR}
\left\langle \PP^1(v_1\ff^1),\iL \PP^1(v_1\ff^0)\right\rangle_I\,\xd x\,\xd t \\
\hat L_4 &=  \int_0^T\!\!\!\int_{\bR}
\left\langle  \mN(\ff),\iL \PP^1(v_1\WW^0)\right\rangle_I\,\xd x\,\xd t.
\end{align*}
The term $\hat L_2$ in \eqref{e:debutestW0} is treated in the same way as the corresponding term $L_2$ in Subsections \ref{sss:L2} and \ref{ss:P1Dxf0}. Indeed, we can prove a result of the same kind as Lemma~\ref{Lemma Yless 3.3} involving $\ff_0$ instead of $\pa_x\ff^0$, by using, among others properties, \eqref{CL Q1}--\eqref{CL ELL}. Hence, we can write, for any $\theta\in (0,\theta_0)$, 
\begin{equation*}
\hat L_2
\ge \lambda_2\Gnm \int_0^T\!\!\!\int_{\bR}\|\ff^0\|_I^2\,\xd x\,\xd t
- \lambda_2\theta \,C_2 \int_0^T\!\!\!\int_{\bR} \|\ff^1\|_I^2\,\xd x\,\xd t 
+ 2\lambda_2\theta\int_\bR Q^1 \,\ell\,\xd x\Big|_0^T.
\end{equation*}
The term $\hat L_3$ has the same structure as $L_3$ in \ref{sss:L3}. Consequently, involving the same constant $K_3$ as in \eqref{e:esttmpL3}, we have
$$|\hat L_3| \le K_3 \int_0^T\!\!\!\int_{\bR}
\left\|\ff^0\right\|_I\left\|\ff^1\right\|_I\,\xd x\,\xd t.$$
The term $\hat L_4$ with the nonlinear operator is also treated as $L_4$ in \ref{sss:L4}
to yield
\begin{equation*}
|\hat L_4| \le K_4 \int_0^T\!\!\!\int_{\bR} \left\|\WW^0 \right\|_I 
\left(\|\ff^0\|_I^2+ \left\|(1+|v|)^{1/2}\ff^1 \right\|_I^2\right)\,\xd x\,\xd t.
\end{equation*}
The main difference with Subsection \ref{ss:1ststepf0} comes from the treatment of $\hat L_1$. Integrating by parts with respect to $t$, and then replacing $\pa_t \WW^0$ by its expression in \eqref{eq W0}, yields
\begin{equation*}
|\hat L_1| \le \int_0^T\!\!\!\int_{\bR}\left\|\ff^1\right\|_I  
 \left\|\iL \PP^1\left(v_1\PP^0(v_1 \ff)\right)\right\|_I \,\xd x\,\xd t
 +\left|\int_{\bR}\left\langle \iL\ff^1, v_1\WW^0\right\rangle_I\,\xd x ~\Big|_0^T \right|.
\end{equation*}
Thanks to the boundedness of $\iL$ and the norm equivalence argument on $\bP^0$ \eqref{Ceq}, and computing directly $\PP^0(v_1 \ff)$ as in \eqref{Cchi}, we obtain 
$$\left\|\iL \PP^1\left(v_1\PP^0(v_1 \ff)\right)\right\|_I \le 
\Cbi \CkerL \Cki  \|\ff\|_I.$$
Setting $\hat K_1= \Cbi \CkerL \Cki >0$, we get
\begin{equation*}
|\hat L_1| \le \hat K_1 \int_0^T\!\!\!\int_{\bR}\left\|\ff^1\right\|_I  
 \left\|\ff\right\|_I \,\xd x\,\xd t
 +\left|\int_{\bR}\left\langle \iL\ff^1, v_1\WW^0\right\rangle_I\,\xd x ~\Big|_0^T \right|.
\end{equation*}
Let us sum up the situation on $\WW^0$. Taking into account the estimates on 
$\hat L_1$, $\hat L_2$, $\hat L_3$, $\hat L_4$ in \eqref{e:debutestW0}, and applying Young's inequality in the estimates on $\hat L_1$ and $\hat L_3$ with a parameter $\delta_1>0$ to be chosen later, we get
\begin{multline}\label{e:finestW0}
\frac12 \int_\bR \|\WW^0\|_I^2\,\xd x \Big|_0^T
+\lambda_2 \Gnm\int_0^T\!\!\!\int_{\bR}\|\ff^0\|_I^2\,\xd x\,\xd t
- \theta \,\lambda_2C_2 \int_0^T\!\!\!\int_{\bR} \|\ff^1\|_I^2\,\xd x\,\xd t 
+ 2\lambda_2\theta\int_\bR Q^1 \,\ell\,\xd x\Big|_0^T \\
\le \hat K_1 \int_0^T\!\!\!\int_{\bR} (\frac1{2\delta_1}\left\|\ff^1\right\|_I^2 +  
 \frac{\delta_1}2\left\|\ff\right\|_I^2) \,\xd x\,\xd t 
+ K_3 \int_0^T\!\!\!\int_{\bR}
( \frac{\delta_1}2\left\|\ff^0\right\|_I^2 +  \frac1{2\delta_1}\left\|\ff^1\right\|_I^2)\,\xd x\,\xd t \\
+ K_4 \int_0^T\!\!\!\int_{\bR} \left\|\WW^0 \right\|_I 
\left(\|\ff^0\|_I^2+ 
\left\|(1+|v|)^{1/2}\ff^1 \right\|_I^2\right)\,\xd x\,\xd t
+\left|\int_{\bR}\left\langle \iL\ff^1, v_1\WW^0\right\rangle_I\,\xd x ~\Big|_0^T \right|.
\end{multline}

\subsection{Proof of the global lower-order estimate}
We now carefully mix all the estimates we got so far, starting with the use of the smallness assumption \eqref{smallness ass loe}. We also have to  treat pointwise in time integrals. Recall that all terms at initial time are put in a generic term denoted by $\mI(0)$. Terms at time $T$ must be handled more shrewdly.

First, imposing $\sass\le C_0/(2\KkerL)$, \eqref{e:estfprf1} becomes
\begin{equation}\label{e:estfprf1-small}
\frac 12 \int_{\bR} \left\| \ff\right\|^2_I \, \xd x \Big|_{t=T} 
+ \frac{C_0}2
\int_0^T \int_{\bR} \|(1+|v|)^{1/2}\ff^1\|_I^2 \, \xd x \, \xd t \\
\le \sass\KkerL \int_0^T\!\!\!\int_{\bR} \left\|\ff^0\right\|_I^2\,\xd x \, \xd t +\mI(0).
\end{equation}
 In the same way, if $\sass\le C_0/(3K_5)$, \eqref{e:estDstarfprDstarf1} can be rewritten as 
\begin{equation}\label{e:estDstarfprDstarf1-small}
\frac 12 \int_{\bR} \left\|\pa_\star \ff\right\|^2_I \, \xd x \Big|_{t=T} +
\frac{C_0}2
\int_0^T \int_{\bR} \|(1+|v|)^{1/2}\pa_\star\ff^1\|_I^2 \, \xd x \, \xd t
\le \frac{\sass K_5}2 \int_0^T\!\!\!\int_{\bR}\left\|\pa_x\ff^0\right\|_I^2\, \xd x \, \xd t +\mI(0).
\end{equation}

Let us now deal with the estimates on elements of $\bP^0$. 
In \eqref{e:finestW0}, we can find three terms at time $T$, which are
$$\frac12 \int_\bR \|\WW^0\|_I^2\,\xd x + 2\lambda_2\theta\int_\bR Q^1 \,\ell\,\xd x - \int_{\bR}\left|\left\langle \iL\ff^1, v_1\WW^0\right\rangle_I\right|\,\xd x.$$
We then notice that, if we set $\delta_2=(\Cbi \CkerL)^{2}$, 
$$\int_{\bR}\left|\left\langle \iL\ff^1, v_1\WW^0\right\rangle_I\right|\,\xd x
\le \Cbi \CkerL \int_{\bR} \|\ff^1\|_I \|\WW^0\|_I\,\xd x 
\le \frac 14\int_{\bR}\|\WW^0\|_I^2\,\xd x + \delta_2 \int_{\bR}\|\ff^1\|_I^2\,\xd x.$$
Moreover, we can write
$$-2\int_\bR Q^1 \,\ell\,\xd x
\le \int_\bR(Q^1)^2\,\xd x + \int_\bR\ell^2\,\xd x
\le \int_\bR\|\WW^0\|_I^2\,\xd x +\int_\bR\ell^2\,\xd x.$$
Hence, for $\theta\le(8\lambda_2)^{-1}$, we get 
$$-2\theta\lambda_2\int_\bR Q^1 \,\ell\,\xd x 
\le \frac 18\int_{\bR}\|\WW^0\|_I^2\,\xd x+\theta\lambda_2\int_\bR\ell^2\,\xd x.$$
All in all, the terms at time $T$ in \eqref{e:finestW0} satisfy, for any $0<\theta\le(8\lambda_2)^{-1}$,
\begin{multline*} 
\frac12 \int_\bR \|\WW^0\|_I^2\,\xd x + 2\lambda_2\theta\int_\bR Q^1 \,\ell\,\xd x - \int_{\bR}\left|\left\langle \iL\ff^1, v_1\WW^0\right\rangle_I\right|\,\xd x \\
\ge \frac18 \int_\bR \|\WW^0\|_I^2\,\xd x - \delta_2 \int_{\bR}\|\ff^1\|_I^2\,\xd x - \theta\lambda_2\int_\bR|\ell|^2\,\xd x.
\end{multline*}
Let us set $\hat C_1=\lambda_2 \Gnm$, $\delta_1=\hat C_1(\hat K_1+K_3)^{-1}/2$, $\hat C_2=\theta_0 \lambda_2C_2+\hat K_1+\frac{\hat K_1+K_3}{2\delta_1}+K_4$. Using that $\|\ff\|_I^2 = \|\ff^0\|_I^2+\|\ff^1\|_I^2$ and the previous inequality in \eqref{e:finestW0}, we obtain
\begin{multline}\label{e:finestW0-small}
\frac18 \int_\bR \|\WW^0\|_I^2\,\xd x \Big|_{t=T}
+\frac{\hat C_1}2\int_0^T\!\!\!\int_{\bR}\|\ff^0\|_I^2\,\xd x\,\xd t \\
\le \hat C_2 \int_0^T\!\!\!\int_{\bR} \left\|(1+|v|)^{1/2}\ff^1\right\|_I^2\,\xd x\,\xd t 
+ \theta\lambda_2\int_\bR |\ell|^2\,\xd x\Big|_{t=T}
+ \delta_2 \int_{\bR}\|\ff^1\|_I^2\,\xd x\Big|_{t=T} +\mI(0),
\end{multline}
where we imposed $\sass\le \min(\hat C_1/(4K_4),1)$. 

Eventually, we similarly tackle the estimate \eqref{e:bilanestf0}. We apply \eqref{est P1 v1 dx f0}, and choose $\delta_0=\hat C_1(K_1+K_3)^{-1}/2$, $\hat C_3=\max(\theta_0C_2+K_3/(2\delta_0),K_1/(2\delta_0))$, to obtain
\begin{multline} \label{e:bilanestf0-small}
\frac14 \int_{\bR} \|\ff^0\|_I^2\,\xd x \Big|_{t=T} + 
\frac{\hat C_1}2\int_0^T\!\!\!\int_{\bR} \left\|\pa_x\ff^0\right\|^2_I \xd x\,\xd t\\
\le \hat C_3 \int_0^T\!\!\!\int_{\bR} \left\|(1+|v|)^{1/2}\pa_x\ff^1\right\|_I^2\,\xd x\,\xd t + \hat C_3 \int_0^T\!\!\!\int_{\bR}
\left\|(1+|v|)^{1/2}\pa_t\ff^1\right\|_I^2\,\xd x\,\xd t\\
+\sass K_4 \int_0^T\!\!\!\int_{\bR}\left\|(1+|v|)^{1/2}\ff^1\right\|_I^2
\xd x \,\xd t+ \theta \lambda_2 \int_{\bR} |\pa_x\ell|^2\,\xd x \Big|_{t=T} + \mI(0),
\end{multline}
with $\sass\le \hat C_1/4$.

Now, let us set 
$$\alpha_0=\max\left( 2\delta_2, \frac{4\hat C_2}{C_0}\right)>0,
\qquad \alpha_1 =\frac{4\hat C_3}{C_0}>0.$$
The remaining pointwise in time terms are handled by choosing $\theta$ small enough, say $\theta\le \theta_1$, where $\theta_1>0$ only depends on the problem data, so that 
$$\theta \lambda_2 \int_{\bR} |\ell|^2\,\xd x \Big|_{t=T}
\le \frac18 \int_{\bR} \|\ff^0\|_I^2\,\xd x \Big|_{t=T}, \qquad
\theta\lambda_2\int_\bR |\pa_x\ell|^2\,\xd x\Big|_{t=T} \le 
\frac{\alpha_1}4 \int_{\bR} \left\|\pa_x\ff\right\|^2_I \,\xd x \Big|_{t=T}.$$
We then multiply \eqref{e:estfprf1-small} by $\alpha_0$ and both estimates \eqref{e:estDstarfprDstarf1-small} (for derivatives in $t$ and $x$) by $\alpha_1$, to jointly add them to \eqref{e:finestW0-small}--\eqref{e:bilanestf0-small}. Assuming moreover that 
$$\sass \le \min\left(
\frac{\hat C_2}{2K_4}, \frac{\hat C_1}{2\alpha_1 K_5}, 
\frac{\hat C_1}{4\alpha_0\KkerL} \right),$$
we can finally write
\begin{multline*}
\frac 18 \int_\bR \|\WW^0\|_I^2\,\xd x \Big|_{t=T}
+\frac 18 \int_\bR \|\ff^0\|_I^2\,\xd x \Big|_{t=T}
+\frac{\alpha_0}4 \int_{\bR} \left\| \ff\right\|^2_I \, \xd x \Big|_{t=T}
+\frac{\alpha_1}4 \int_{\bR} \left\|\pa_x\ff\right\|^2_I \,\xd x\Big|_{t=T} \\
+\frac{\alpha_1}2 \int_{\bR} \left\|\pa_t\ff\right\|^2_I \,\xd x \Big|_{t=T}
+\frac{\hat C_1}4\int_0^T\!\!\!\int_{\bR}\|\ff^0\|_I^2\,\xd x\,\xd t
+\frac{\hat C_1}4\int_0^T\!\!\!\int_{\bR} \left\|\pa_x\ff^0\right\|^2_I \xd x\,\xd t \\
+\frac{\hat C_2}2
\int_0^T\!\!\!\int_{\bR} \left\|(1+|v|)^{1/2}\ff^1\right\|_I^2 \, \xd x \, \xd t 
+\hat C_3
\int_0^T\!\!\!\int_{\bR} \left\|(1+|v|)^{1/2}\pa_x\ff^1\right\|_I^2 \, \xd x \, \xd t\\
+\hat C_3
\int_0^T\!\!\!\int_{\bR} \left\|(1+|v|)^{1/2}\pa_t\ff^1\right\|_I^2 \, \xd x \, \xd t
\le \mI(0),
\end{multline*}
which yields \eqref{loe}. To summarize on $\sass$, we set 
$$\eps_0 = \min\left(
\frac{\hat C_2}{2K_4}, \frac{\hat C_1}{2\alpha_1 K_5}, 
\frac{\hat C_1}{4\alpha_0\KkerL}, \frac{\hat C_1}{4}, \frac{\hat C_1}{4K_4},
\frac{C_0}{3K_5}, \frac{C_0}{2\KkerL},1
\right).$$
The previous inequality holds as soon as $\sass\le \eps_0$. This concludes the proof of Proposition~\ref{prop:loe}.

\subsection{Comments on Assumption \eqref{e:hypintf0} on $\ff^0$} \label{ss:hypW0}
Let us explain how to handle the computations when Assumption~\eqref{e:hypintf0} does not hold. In this case, we know from \eqref{e:intrho}--\eqref{e:inte} that $\int_\bR \rho_i\,\xd x$, $\int_\bR q^k\,\xd x$ and $\int_\bR e\,\xd x$ remain constants with respect to $t$, one of them, at least, then being nonzero. Consequently, one of the antiderivatives of the macroscopic perturbation would not lie in $L^2(\bR)$. To deal with that issue, we consider a nonnegative $C^\infty$-compactly supported function $\psi$ in the variable $x$ satisfying $\int_\bR\psi(y)\,\xd y=1$, and set
$$\Psi(x)=\int_{-\infty}^x\psi(y)\,\xd y, \qquad \FF^0_{\mbox{\tiny in}}(v)=\int_\bR \ff^0(0,y,v)\,\xd y,\qquad \widetilde\WW^0(t,x,v)=\WW^0(t,x,v)-\Psi(x) \FF^0_{\mbox{\tiny in}}(v).$$
As a function of $v$, $\widetilde\WW^0$ clearly belongs to $\bP^0$. And now $\widetilde\WW^0$ must replace $\WW^0$ in any computations involving its $L^2$ norm. We just have to track the changes implied by this replacement in Subsection~\ref{ss:estW0}. Since $\pa_t\widetilde\WW^0=\pa_t\WW^0$, \eqref{e:debutestW0} becomes, using integration by parts,
\begin{multline} \label{e:debutestW0tilde}
\frac12 \int_\bR \|\widetilde\WW^0\|_I^2\,\xd x \Big|_0^T 
+ \tilde L_1+ \hat L_2+ \hat L_3+ \tilde L_4 
+\int_0^T\!\!\!\int_{\bR} \psi \left\langle \iL \PP^1(v_1\ff),\PP^1(v_1\FF^0_{\mbox{\tiny in}})\right\rangle_I\,\xd x\,\xd t \\
-\int_0^T\!\!\!\int_{\bR} \Psi\left\langle\PP^0(v_1\FF^0_{\mbox{\tiny in}}),\ff^0 \right\rangle_I\,\xd x\,\xd t 
+\int_0^T\!\!\!\int_{\bR} \psi\Psi\left\langle \PP^0(v_1\FF^0_{\mbox{\tiny in}}),\FF^0_{\mbox{\tiny in}}\right\rangle_I\,\xd x\,\xd t =0,
\end{multline}
where $\tilde L_1$ and $\tilde L_4$ have the same expressions as $\hat L_1$ and $\hat L_4$ with $\widetilde\WW^0$ instead of $\WW^0$, and $\hat L_2$ and $\hat L_3$ remains unchanged, and all are treated in the same way as in Subsection~\ref{ss:estW0}. The change then only lies in the terms involving $\FF^0_{\mbox{\tiny in}}$. The last one does not depend on $\ff$ and its absolute value equals
$$\frac T2 \left|\left\langle \PP^0(v_1\FF^0_{\mbox{\tiny in}}),\FF^0_{\mbox{\tiny in}}\right\rangle_I\right|
\le C \|\FF^0_{\mbox{\tiny in}}\|_I^2,$$
which can be considered as a contribution to $\mI(0)$. 
Since $\psi\in L^2(\bR)$, the remaining term involving it is easily upper-bounded using Young's inequality by 
$$\kappa \int_0^T\!\!\!\int_{\bR} \|\ff\|_I^2\,\xd x\,\xd t + C \|\FF^0_{\mbox{\tiny in}}\|_I^2,$$
where $\kappa$ can be chosen as small as needed. The last remaining term (the one with $\Psi$) is handled by writing the equation satisfied by $\int_\bR\Psi \ff^0\,\xd x$. It is obtained by multiplying \eqref{equation for f0} by $\Psi$ and integrating with respect to $x\in\bR$ and $s\in[0,t]$: 
$$\int_\bR\Psi \ff^0(t,x,v)\,\xd x=\int_\bR\Psi \ff^0(0,x,v)\,\xd x+\int_0^t\!\!\!\int_\bR \psi \PP^0(v_1\ff)\,\xd x\,\xd s.$$
It implies that, for some constant $K>0$ depending on $\psi$ and $T$, 
$$\left\|\int_\bR\Psi \ff^0(t,x,\cdot)\,\xd x\right\|_I^2
\le K \left[\left\|\int_\bR\Psi \ff^0(0,x,\cdot)\,\xd x\right\|_I^2
+ \int_0^T\!\!\!\int_{\bR} \|\PP^0(v_1\ff)\|_I^2\,\xd x\,\xd t\right].$$
The first term on the right-hand side of the previous estimate is again a contribution to $\mI(0)$. 
Consequently, using Young's inequality, we have 
$$\left|\int_0^T\!\!\!\int_{\bR} \Psi\left\langle\PP^0(v_1\FF^0_{\mbox{\tiny in}}),\ff^0 \right\rangle_I\,\xd x\,\xd t\right| = 
\left|\int_0^T\left\langle \PP^0(v_1\FF^0_{\mbox{\tiny in}}), \int_{\bR} \Psi \ff^0\,\xd x \right\rangle_I\,\xd t\right|\le \mI(0) + \kappa \int_0^T\!\!\!\int_{\bR} \|\ff\|_I^2\,\xd x\,\xd t,$$
where $\kappa$ can again be chosen as small as needed.
Both contributions upper-bounded by $\kappa$-multiplied terms in \eqref{e:debutestW0tilde} can then be absorbed by similar counterparts on the left-hand side of \eqref{loe}.

\section{Elements of proof for the higher-order estimate} \label{s:hoe}
Unlike what we did in the previous section about the lower-order estimate, for the sake of simplicity, we shall now use the corresponding smallness assumption \eqref{smallness ass hoe} as soon as possible in our computations for the higher-order estimate. 

\subsection*{Estimate involving $\pa^p\ff^1$, $1\le |p|\le 5$}
We take the $\pa^p$ derivative of the linearized Boltzmann equation \eqref{equation for f}, scalarily multiply it by $\pa^p \ff$ and integrate with respect to $t$ and $x$, to obtain
\begin{equation}\label{e:debutestimDpf}
\frac 12 \int_{\bR} \left\| \pa^p \ff \right\|_I^2 \, \xd x \Big{|}_0^T -  \int_{0}^T\!\!\!\int_{\bR} \left\langle \mL \pa^p\ff, \pa^p\ff \right\rangle_I \xd x \, \xd t 
 = \int_{0}^T\!\!\!\int_{\bR} \left\langle \pa^p \mN(\ff), \pa^p\ff \right\rangle_I \, \xd x \, \xd t.
\end{equation}
In the second integral, we use the spectral gap property \eqref{spectral gap} of $\mL$, so that
\begin{equation*}
- \int_{0}^T\!\!\!\int_{\bR}
\left\langle \mL \pa^p\ff, \pa^p\ff \right\rangle_I\xd x \, \xd t \ge
C_0 \int_{0}^T\!\!\!\int_{\bR} 
\left\| (1+|v|)^{1/2}\pa^p\ff^1  \right\|^2_I\xd x \, \xd t.
\end{equation*}
Let us focus on the term with the nonlinear operator $\mN$. We first notice that 
$$\pa^p \mN(\ff)  = \M^{-1/2}\sum_{0\le |p'|\le |p|} \binom{|p|}{|p'|}
\QQ\left(\M^{1/2} \pa^{p'}\ff, \M^{1/2}\pa^{p-p'}\ff\right).$$
Thanks to Lemma \ref{prop sulem}, we get
\begin{equation}\label{e:estimDpNf}
\left\|(1+|v|)^{-1/2}\pa^p \mN(\ff) \right\|_I \le \Cnmbet
\sum_{0\le |p'|\le |p|} \binom{|p|}{|p'|} \left\| (1+|v|)^{1/2} \pa^{p'}\ff \right\|_I \left\| (1+|v|)^{1/2} \pa^{p-p'}\ff\right\|_I.
\end{equation}
For $3\le|p|\le 5$, if the sub-index $p'$ of $p$ satisfies $|p'|\le 2$, the smallness assumption \eqref{smallness ass hoe} can be used for the term $\| (1+|v|)^{1/2} \pa^{p'}\ff\|_I$, and if $3\le |p'|\le 5$, then $|p-p'|\le 2$, and the term $\| (1+|v|)^{1/2} \pa^{p-p'}\ff\|_I$ can be handled through \eqref{smallness ass hoe}. Hence, in any case (when $|p|\le 2$, it is straightforward), we can write, for some constant $C_p>0$ only depending on $\mm$, $\nn$, $\beta$ and $p$,
$$\left\|(1+|v|)^{-1/2}\pa^p \mN(\ff) \right\|_I \le C_p\,\sass
\sum_{|p'|\le |p|} \left\| (1+|v|)^{1/2} \pa^{p'}\ff \right\|_I.$$
Consequently, we have, for any $p$, $1\le |p|\le 5$,
$$\int_{0}^T\!\!\!\int_{\bR} \left|
\left\langle \pa^p \mN(\ff), \pa^p\ff \right\rangle_I\right|\,\xd x\,\xd t 
\le C_p\,\sass \int_{0}^T\!\!\!\int_{\bR}\sum_{|p'|\le |p|} 
\left\| (1+|v|)^{1/2} \pa^{p'}\ff \right\|_I
\left\| (1+|v|)^{1/2} \pa^{p}\ff \right\|_I\,\xd x\,\xd t.$$

All in all, summing \eqref{e:debutestimDpf} for all indices $p$, $1\le |p|\le 5$, we get, for $\sass$ small enough,
\begin{multline}\label{e:finestimDpf}
\frac 12 \sum_{1\le|p|\le 5}\int_{\bR} \left\| \pa^p \ff \right\|_I^2 \, \xd x \Big{|}_{t=T} +
\frac{C_0}2 \sum_{1\le|p|\le 5}\int_{0}^T\!\!\!\int_{\bR} 
\left\|(1+|v|)^{1/2} \pa^p\ff^1\right\|_I^2 \xd x \, \xd t \\
\le C\,\sass\sum_{1\le|p|\le 5}\int_{0}^T\!\!\!\int_{\bR}
\left\|\pa^p\ff^0\right\|_I^2 \xd x \, \xd t ~~+\mI(0),
\end{multline}
where $C>0$ only depends on the data of the problem. 

\subsection*{Estimate involving $\pa^r\ff^0$, $1\le |r|\le 4$}
We take the $\pa^r$ derivative of \eqref{equation for f0} and observe that $\pa^r\ff^0$ satisfies exactly the same kind of equation as $\ff^0$ itself. Consequently, \eqref{e:bilanestf0} also holds for $\pa^r\ff^0$ instead of $\ff^0$, with some changes only on the nonlinear term $\hat L_4^r$, 
\begin{multline} \label{e:bilanestDrf0}
\frac12 \int_{\bR} \|\pa^r\ff^0\|_I^2\,\xd x \Big|_0^T + 
\lambda_2\int_0^T\!\!\!\int_{\bR} \left\|\PP^1(v_1\pa_x\pa^r\ff^0)\right\|^2_I \xd x\,\xd t\\
\le K_1 \int_0^T\!\!\!\int_{\bR}
(\frac1{2\delta_0}\left\|\pa_t\pa^r\ff^1\right\|_I^2 + \frac{\delta_0}2 \left\|\pa_x\pa^r\ff^0\right\|_I^2)\xd x\,\xd t\\
+ K_3 \int_0^T\!\!\!\int_{\bR}
(\frac{\delta_0}2\left\|\pa_x\pa^r\ff^0\right\|_I^2 + \frac1{2\delta_0}\left\|\pa_x\pa^r\ff^1\right\|_I^2 )\xd x\,\xd t
+|{\hat L_4^r}|,
\end{multline}
where 
$${\hat L_4^r}=-\int_{0}^T\!\!\! \int_{\bR} \left\langle v_1 \pa_x \iL \pa^r\mN(\ff),\pa^r\ff^0\right\rangle_I \xd x \, \xd t
=\int_{0}^T\!\!\! \int_{\bR} \left\langle \pa^r\mN(\ff),\iL \PP^1(v_1 \pa_x  \pa^r\ff^0)\right\rangle_I \xd x \, \xd t.
$$
Using the smallness assumption \eqref{smallness ass hoe} and \eqref{e:estimDpNf}, we get
\begin{equation}\label{e:estimLhatr}
|\hat L_4^r| \le C_r K_3 \sass
\left[\sum_{|p'|\le |r|} \left\|(1+|v|)^{1/2} \pa^{p'}\ff\right\|_I \right]
\left\|\pa_x\pa^r\ff^0\right\|_I.
\end{equation}
The second term in \eqref{e:bilanestDrf0} is handled thanks to Lemma \ref{Lemma Yless 3.3}, leading to
\begin{multline}\label{e:estP1vDxDrf0}
\int_0^T\!\!\!\int_{\bR}\| \PP^1(v_1 \pa_x \pa^r \ff^0)\|_I^2\,\xd x\,\xd t \ge 
\Gnm \int_0^T\!\!\!\int_{\bR} \|\pa_x \pa^r\ff^0\|_I^2\,\xd x\,\xd t \\
- \theta C_2 \int_0^T\!\!\!\int_{\bR} \|\pa_x \pa^r\ff^1\|_I^2\,\xd x\,\xd t
+2\theta \int_{\bR} \pa^rq\,\pa_x\pa^r \ell\,\xd x\Big{|}_0^T.
\end{multline}

\subsection*{Proof of Theorem \ref{prop higher order estimate}}
Following the same reasoning as in the previous section, we combine the lower-order estimate \eqref{loe}, the estimate \eqref{e:finestimDpf} on the derivatives of $\ff^1$, and the estimate on the derivatives of $\pa_x\ff^0$ obtained from \eqref{e:bilanestDrf0}--\eqref{e:estP1vDxDrf0}. 
In the right-hand side of estimate \eqref{e:finestimDpf}, the treatment of the pure time derivatives $\partial_t^p \ff^0$ is done as in Subsection~\ref{ss:estDstarf1}. They are thus controlled by terms which involve at least one space derivative $\partial_x \partial_t^{p-1} \ff$, still multiplied by $\sass$. In the end, choosing $\theta$ and $\sass$ small enough, the estimate \eqref{e:hoefinal} is proved.

\appendix

\section{Estimate on $\QQ$}\label{appendix-Q}

In order to deal with the terms involving $\mN(\ff)$, we need the following result. It requires the hard-sphere assumption \eqref{ass Bij L2} and its proof is provided here, despite its similarity to the one in \cite[Lemma A.1]{gol-per-sul} (see also \cite[Lemma B.1]{liu-yu04}), for the sake of completeness.
\begin{lemma}\label{prop sulem}
Assuming that all the cross sections satisfy the hard-sphere assumption \eqref{ass Bij L2}, there exists $\Cnmbet>0$, only depending on $\mm$, $\nn$ and the cross sections, such that, for any $\ff$, $\gg\in \bD$, 
\begin{equation}\label{e:sulem}
\left\| (1+|v|)^{-1/2} \M^{-1/2} \QQ\left(\M^{1/2} \ff, \M^{1/2} \gg \right)  \right\|_I
\le \Cnmbet \left\| (1+|v|)^{1/2} \ff \right\|_I \left\| (1+|v|)^{1/2} \gg \right\|_I.
\end{equation} 
\end{lemma}
\begin{proof}
Denote by $A$ the left-hand side of \eqref{e:sulem}. We can write, thanks to the Cauchy-Schwarz inequality and the hard-sphere assumption \eqref{ass Bij L2} on the cross sections,
\begin{multline*}
A^2= \sum_i \int_{\bR^3} (1+|v|)^{-1} (n_iM_i)^{-1} \, Q_i\left(\M^{1/2} \ff, \M^{1/2} \gg\right)^2\,\xd v\\
\le I \sum_{i,j} n_j \int_{\bR^3} (1+|v|)^{-1}
\left[ \iint_{\bR^3\times S^2} M_j(v_*)^{1/2}\left( f_i(v') g_j(v_*') -f_i(v) g_j(v_*) \right) \mathcal B_{ij}\,\xd\omega\,\xd v_*\right]^2\,\xd v \\
\le 2I \sum_{i,j} n_j \int_{\bR^3} \frac{\beta_{ij}}{1+|v|}
\left[ \left(\iint_{\bR^3\times S^2} M_j(v_*)^{1/2} f_i(v') g_j(v_*') |(v-v_*)\cdot\omega|\,\xd\omega\,\xd v_*\right)^2 \right.\\
\left.+\left(\iint_{\bR^3\times S^2} M_j(v_*)^{1/2} f_i(v) g_j(v_*) |(v-v_*)\cdot\omega|\,\xd\omega\,\xd v_*\right)^2\right]\,\xd v.
\end{multline*}
Denote $\beta=\max \beta_{ij}>0$. Noticing that 
$$|(v-v_*)\cdot\omega| \le  (|v|+|v_*|),$$
and that $v_*\mapsto M_j(v_*)^{1/2} |v_*|$ is a bounded function on $\bR^3$, there exists a constant $\Cnbet>0$, only depending on $\nn$ and $\beta$, such that $A$ is upper-bounded by
\begin{equation*}
\Cnbet \sum_{i,j}  
\left[ \iiint_{\bR^3\times\bR^3\times S^2} f_i(v')^2 g_j(v_*')^2(1+|v|)\,\xd\omega\,\xd v_*\,\xd v + \iint_{\bR^3\times\bR^3} f_i(v)^2 g_j(v_*)^2 (1+|v|)\,\xd v_*\,\xd v.
\right] 
\end{equation*}
Let us focus on the first addend, since the second one clearly equals $\Cnbet\left\|(1+|v|)^{1/2}\ff\right\|_I^2 \|\gg\|_I^2$ thanks to the Fubini theorem. We perform the change of variables $(v,v_*,\omega)\mapsto(v',v_*',\omega)$ in the integral, which becomes
$$\iiint_{\bR^3\times\bR^3\times S^2} f_i(v)^2 g_j(v_*)^2(1+|v'|)\,\xd\omega\,\xd v_*\,\xd v.$$
The collision rules \eqref{collisionrule} then ensure that, for some constant $\Cm>0$ only depending on $\mm$ (uniform with respect to the indices $i$ and $j$),
$$|v'| \le |v| + \frac{2 m_j}{m_i + m_j} |v_*| \le \Cm (|v'| + |v'_*|), \qquad \forall v,v_*\in\bR^3,$$
so that, up to a value change of $\Cm>0$, still only depending on $\mm$,
$$1+|v'| \leq \Cm (2+|v|+|v_*|), \qquad \forall v,v_*\in\bR^3.$$ 
The previous inequality allows to upper-bound the first addend in the same way as the second one thanks to the Fubini theorem, which concludes the proof.
\end{proof}

\section{Conservation laws}\label{s:conslaws} 
In this section, we provide the proof of Proposition~\ref{CL prop}. We take the projected Boltzmann equation \eqref{equation for f0} satisfied by $\ff^0$ and scalarily multiply it by each $\kki^i \M^{1/2}$, $1\le i\le I+4$. 

We first choose $1\le i \le I$ to obtain the conservation law for $\rho_i$, that is
$$
\pa_t \left\langle \ff^0, \M^{1/2} \kki^i \right\rangle_I  +  
\left\langle \PP^0(v_1 \pa_x \ff^0), \M^{1/2} \kki^i \right\rangle_I + 
\left\langle \PP^0(v_1 \pa_x \ff^1), \M^{1/2} \kki^i \right\rangle_I =0.
$$
The first term obviously gives $\pa_t\rho_i$. The second term can of course be rewritten as
$$\left\langle \PP^0(v_1 \pa_x \ff^0), \M^{1/2} \kki^i \right\rangle_I = 
\left\langle v_1 \pa_x \ff^0, \M^{1/2} \kki^i  \right\rangle_I,$$
which we can explicitly compute, helped by parity arguments,
$$\left\langle v_1 \pa_x \ff^0, \M^{1/2} \kki^i \right\rangle_I 
=\pa_x\left[\int_{\bR^3} v_1 f^0_i (n_iM_i)^{1/2}\xd v\right]
=\frac{n_i \pa_x q^1}{\sqrt{\sum n_jm_j}} \int_{\bR^3}  m_i v_1^2 M_i \xd v
=\frac{n_i \pa_x q^1}{\sqrt{\sum n_jm_j}}.
$$
The third term being unchanged, we recover \eqref{CL rhoi}.

To obtain the conservation law on $q^1$, we can write
\begin{equation*}
\pa_t q^1  +  
\left\langle \PP^0( v_1 \pa_x \ff^0), \M^{1/2} \kki^{I+1} \right\rangle_I 
+\left\langle \PP^0(v_1 \pa_x \ff^1), \M^{1/2} \kki^{I+1} \right\rangle_I =0.
\end{equation*}
The second term also simplifies thanks to parity arguments
\begin{multline*}
\left\langle v_1 \pa_x \ff^0, \M^{1/2} \kki^{I+1} \right\rangle_I 
= \sum_{i} \pa_x \left[\int_{\bR^3} v_1 f^0_i (n_iM_i)^{1/2} \chi^{I+1}_i 
\xd v\right] \\ 
= \sum_{i} \frac{m_i}{\sqrt{\sum n_jm_j}}~
\pa_x \int_{\bR^3} \left(\rho_i {v_1}^2 \sqrt{n_i} M_i 
+ \frac{e}{\sqrt{6\sum n_j}} (m_i |v|^2-3) {v_1}^2 n_i M_i\right) \xd v\\ 
= \frac{1}{\sqrt{\sum n_jm_j}} 
\left( \sum_i \sqrt{n_i}\pa_x \rho_i + \frac{\sqrt{6\sum n_j}}{3} \pa_x e \right). 
\end{multline*}
This finishes the derivation of the conservation law \eqref{CL q1} for $q^1$. The one for $q^2$ and $q^3$ \eqref{CL q}, is then straightforward.

Finally, the conservation law  \eqref{CL e} for $e$ reads
\begin{equation*}
\pa_t e  +  
\left\langle \PP^0( v_1 \pa_x \ff^0), \M^{1/2} \kki^{I+4} \right\rangle_I 
+\left\langle \PP^0(v_1 \pa_x \ff^1), \M^{1/2} \kki^{I+4} \right\rangle_I =0.
\end{equation*}
Let us compute the second term of the above equation, which gives
\begin{multline*}
\left\langle v_1 \pa_x \ff^0, \M^{1/2} \kki^{I+4} \right\rangle_I 
= \sum_{i=1}^I \int_{\mathbb{R}^3} v_1  \partial_x  f^0_i (n_iM_i)^{1/2}  
\chi^{I+4}_{i} \xd v \\ 
= \frac{1}{\sqrt{6\sum n_j}} \sum_i n_i \pa_x \left[ \int_{\bR^3} 
\left(m_i |v|^2 -3\right) \left( \frac{q^1}{\sqrt{\sum n_jm_j}} m_i  v_1 \right) v_1 M_i \xd v \right]\\
=\sum_i  \frac{2n_i}{\sqrt{6\sum n_j}} \frac{\pa_x q^1}{\sqrt{\sum n_jm_j}} = 
\frac 13 \sqrt{\frac{6\sum n_j}{\sum n_jm_j}} \ \pa_x q^1.
\end{multline*}
Summarizing, the conservation law \eqref{CL e} for $e$ is obtained.

\section{Proof of Lemma \ref{Lemma Yless 3.3} in three dimensions}\label{s:korn}

In what follows, the one-dimensional notations are straightforwardly extended in three dimensions. The only significant difference with respect to the one-dimensional case is the treatment of $\na_x q$, which requires the use of the Korn inequality to conclude. Let us rewrite the main equalities and estimates in the three-dimensional setting. We first note that the conservation laws on the fluid quantities in $\ff^0$ still hold, \ie
\begin{align*}
&\pa_t \rho_i + \frac{n_i}{\sqrt{\sum n_jm_j}} \na_x\cdot q 
+ \left\langle \PP^0(v \cdot \na_x \ff^1), \kki^i \M^{1/2} \right\rangle_I =0, 
\quad 1\le i\le I, \\
&\pa_t q^k + \na_x \ell + 
\left\langle \PP^0(v \cdot \na_x  \ff^1), \kki^{I+1} \M^{1/2} \right\rangle_I =0, 
\qquad k=1,\,2,\,3,\\
&\pa_t e + \frac 13 \sqrt{\frac{6\sum n_j}{\sum n_jm_j}} \na_x \cdot q
+ \left\langle \PP^0(v \cdot \na_x  \ff^1), \kki^{I+4} \M^{1/2} \right\rangle_I =0, \\
&\pa_t \ell + {\frac{\displaystyle\sum {n_j}^{3/2}+\frac 23 \sum n_j}{\displaystyle\sum n_j m_j}} ~\na_x \cdot q \\
& \qquad \qquad \qquad +\frac 1{\sqrt{\sum n_j m_j}}\left\langle \PP^0(v\cdot \na_x \ff^1), \left(\sum_i \sqrt{n_i}\kki^i+\frac 13 \sqrt{6\sum n_j}\kki^{I+4}\right) \M^{1/2}\right\rangle_I =0,
\end{align*}
where $\ell$ is still defined by \eqref{ell}. Of course, we follow the strategy of  Section \ref{ss:P1Dxf0}. We first compute
\begin{multline*} 
\|v\cdot \na_x \ff^0\|_I^2 
= \sum_i \frac{1}{m_i} 
\left|\na_x \rho_i + \frac{2\sqrt{n_i}}{\sqrt{6 \sum n_j}} \na_x e \right|^2
+ \frac{5}{3} \frac{1}{\sum n_j} \sum_i \frac{n_i}{m_i}|\na_x e|^2 \\
+\frac{\sum n_j}{\sum n_j m_j} 
\left[\left(\na_x \cdot q\right)^2+\sum_{l,m}\left((\pa_{x_l} q^m)^2 + \pa_{x_l} q^m \pa_{x_m} q^l \right)\right].
\end{multline*}	
In the same way, as in \eqref{P0 v grad f0}, we can write 
\begin{equation*}
\| \PP^0 (v\cdot\na_x \ff^0) \|_I^2 = |\na_x \ell|^2 
+ \frac 53\frac{\sum n_j}{\sum n_jm_j} (\na_x \cdot q)^2.
\end{equation*}
Then, for some $\theta\in (0,1)$, \eqref{est P1 v1 dx f0 with theta} becomes
\begin{multline*} 
\|\PP^1(v\cdot\na_x \ff^0)\|_I ^2 \ge
\theta\sum_i \frac{1}{m_i} \left|\na_x \rho_i + \frac{2\sqrt{n_i}}{\sqrt{6\sum n_j}} \na_x e \right|^2 
+\frac{5}{3} \frac{1}{\sum n_j} \left( \sum \frac{n_j}{m_j}\right) |\na_x e|^2 \\
+ \frac{\sum n_j}{\sum n_j m_j} \left[-\frac23\left(\na_x \cdot q\right)^2+\sum_{l,m}\left((\pa_{x_l} q^m)^2 + \pa_{x_l} q^m \pa_{x_m} q^l \right)\right]
- \theta |\na_x \ell|^2.
\end{multline*}
Besides, using the three-dimensional laws on $q^k$ and $\ell$, \eqref{theta dx ell} can be rewritten into
\begin{multline*}
\theta \int_0^T\!\!\!\int_{\bR^3}|\na_x \ell|^2\,\xd x\,\xd t
+ 2\theta\int_{\bR^3} q\cdot \na_x \ell\,\xd x \Big|_0^T \\
\le 2\theta \left(\frac{\sum n_j^{3/2} + \frac23 \sum n_j}{\sum n_j m_j} + \frac12 \right)\int_0^T\!\!\!\int_{\bR^3}(\na_x \cdot q)^2\,\xd x\,\xd t 
+ \theta C_\kki \int_0^T\!\!\!\int_{\bR^3} \|\na_x \ff^1\|_I^2\,\xd x\,\xd t.
\end{multline*}
Thanks to the Korn inequality, for $\theta$ small enough, there exists $\Cth>0$, only depending on $\theta$, such that 
$$-\left(\frac23+\theta\right) \int_{\bR^3}\left(\na_x \cdot q\right)^2\,\xd x
+\sum_{l,m}\int_{\bR^3}\left((\pa_{x_l} q^m)^2 + \pa_{x_l} q^m \pa_{x_m} q^l \right)\,\xd x \ge \Cth \int_{\bR^3} |\na_x q|^2\,\xd x.$$
Eventually, in the same way as in Subsection~\ref{ss:P1Dxf0}, the norm equivalence argument on $\bP^0$ \eqref{Ceq} allows to obtain
\begin{multline*}
\int_0^T\!\!\!\int_{\bR^3}\| \PP^1(v\cdot \na_x \ff^0)\|_I^2\,\xd x\,\xd t \ge 
\Gnm \int_0^T\!\!\!\int_{\bR^3} \|\na_x \ff^0\|_I^2\,\xd x\,\xd t \\
- \theta C_2 \int_0^T\!\!\!\int_{\bR^3} \|\na_x \ff^1\|_I^2\,\xd x\,\xd t
+2\theta \int_{\bR^3} q\cdot\na_x \ell\,\xd x\Big{|}_0^T,
\end{multline*}
for some constants $\Gnm>0$ depending on $\nn$, $\mm$ and $\theta$, and $C_2>0$ depending on $\nn$ and $\mm$.

\bibliography{biblioliuyu}
\end{document}